\documentclass[12pt]{amsart}

\usepackage{amsmath,amsthm,amssymb,mathrsfs,amsfonts,verbatim}
\usepackage{enumerate}
\usepackage{etoolbox} 
\usepackage{bbm}
\usepackage[all,tips]{xy}
\usepackage{graphicx,ifpdf}
\ifpdf
\fi

\newtheorem{thm}{Theorem}[section]
\newtheorem{lem}[thm]{Lemma}
\newtheorem{cor}[thm]{Corollary}
\newtheorem{prop}[thm]{Proposition}

\theoremstyle{definition}

\theoremstyle{remark}

\newcommand{\R}{\mathbb{R}}
\newcommand{\N}{\mathbb{N}}
\newcommand{\SL}{\operatorname{SL}}
\newcommand{\GL}{\operatorname{GL}}
\newcommand{\Z}{\mathbb{Z}}

\newcommand{\dd}{\; \mathrm{d}}

\newcommand{\Ad}{\operatorname{Ad}}

\numberwithin{equation}{section}

\newcommand{\Rmnum}[1]{\expandafter\@slowromancap\romannumeral #1@}

\begin{document}

\title{Pointwise equidistribution  for    one parameter diagonalizable group action on homogeneous space}

\author{Ronggang Shi}
\address{School of Mathematical Sciences, Tel Aviv University, Tel Aviv 69978, Israel 
	}
\curraddr{Shanghai Center for Mathematical Sciences, Fudan 
 University, Shanghai 200433, PR China}
 \email{ronggang@fudan.edu.cn}
\thanks{The author is supported by 
 ERC starter grant DLGAPS 279893.}


\subjclass[2000]{Primary   28A33; Secondary 37C85, 37A30.}

\date{}


\keywords{homogeneous dynamics,  equidistribution, ergodic theorem}

\begin{abstract}
\noindent
Let $\Gamma$ be a lattice of a semisimple Lie group $L$. 
 Suppose that one parameter  $\Ad$-diagonalizable subgroup  
$\{g_t\}$ of  $L$ acts ergodically on $L/\Gamma$ 
with respect to the probability Haar measure $\mu$.
For certain proper subgroup $U$ of 
the unstable horospherical subgroup  of $\{g_t\}$ and certain $x\in L/\Gamma$
we show that for almost every  $u\in U$   the trajectory
$\{g_tux: 0\le t\le T\}$ is equidistributed   with respect to $\mu$ as $T\to \infty$.
\end{abstract}

\maketitle

\markright{}

\section{introduction}\label{sec;intro}

Let $(X, \mathcal B,\mu, T)$ be a probability measure preserving system, i.e. 
$\mu$ is a probability measure on the measurable space $(X, \mathcal B)$ and
the measurable map
$T:X\to X$ preserves $\mu$. 
The Birkhoff  ergodic theorem 
says that 
if $T$ is ergodic then  given $f\in L^1_\mu(X)$ 
\begin{equation}\label{eq;birkhoff}
\lim_{N\to \infty}\frac{1}{N} \sum_{n=0}^{N-1}f(T^nx)=\int_Xf\dd \mu
\end{equation}
for  almost every $x\in X$. 

Suppose that   $X$ is a locally compact second countable  
 Hausdorff 
 topological space and $\mathcal B$ is the Borel sigma algebra
of $X$. 
Given $x\in X$ the condition 
 that   (\ref{eq;birkhoff}) holds for  every  $f$ belonging to   the set $C_c(X)$ of 
   continuous functions with compact support
is equivalent to 
\begin{equation}\label{eq;question}
\lim_{N\to \infty}\frac{1}{N} \sum_{n=0}^{N-1}\delta_{T^nx}=\mu
\end{equation}
in the space of finite measures on $X$ under the weak$^*$ topology. 
Here $\delta_y$ denotes the Dirac measure supported on $y\in X$.
A Radon measure $\nu$ on $X$ is said to be $(T,\mu)$ generic if   (\ref{eq;question}) holds
  for $\nu$ almost every $x\in X$. 
 A natural question is whether a measure $\nu$ (usually singular to $\mu$) is $(T,\mu)$ generic.
  
This question is studied by several authors for natural dynamical systems on $X=\mathbb R/\mathbb Z$. 
Let $m,n$ be coprime positive integers greater than or equal to $2$.
Suppose that     $\mu_X$ is the Lebesgue measure on $X$ and  $T_n=\times n $ modulo $\mathbb Z$.
Host \cite{h95} shows that
any  $T_m$ invariant and ergodic probability measure $\nu$ on
$X$ with positive entropy
  is  $(T_n, \mu_X)$ generic. This result  is strengthened by 
 Hochman and  Shmerkin \cite{hs} where they prove that for any $C^2$ diffeomorphism  $\varphi: \mathbb R\to \mathbb R$, the push forward of $\nu$ modulo $\mathbb Z$
is $(T_n, \mu_X)$ generic. The reader can find detailed references of related results in \cite{hs}.

The aim of this paper is to address this question for one parameter $\Ad$-diagonalizable   flows in homogeneous spaces. 
Let $\Gamma$ be a  lattice
of a   Lie group  $L$.
Every subgroup $H$ of $L$
 acts on $L/\Gamma$ by left translations and this action preserves the probability Haar measure $\mu_{L/\Gamma}$. 
 We use $(H, L/\Gamma)$ to denote this measure preserving system. 
 There are two basic types of one parameter subgroups $t\to g_t\in L$ in terms of its image under the adjoint 
representation $\Ad: L\to \GL(\mathfrak l)$ where $\mathfrak l$ is the Lie algebra of $L$.
If  $\Ad(g_t)$   is unipotent, then according to  Ratner's uniform distribution theorem \cite{r912}  
    the Dirac measure $\delta_x$ of any point $x\in L/\Gamma$ is generic 
with respect to some $\{g_t:t\in \mathbb R\}$ ergodic homogeneous probability measure.
If  the 
one parameter subgroup is
 $\Ad$-diagonalizable, i.e.~$\Ad(g_t)$
 is diagonalizable over $\mathbb R$, the unstable horospherical subgroup
 of $g_1$  is defined as 
\[
L^+=\{h\in  L: g_t^{-1}hg_t\to 1_L \mbox{ as } t\to \infty\}
\]
where 
  $1_L$ is the   identity element of the group $L$.
A   variant of Birkhoff ergodic theorem says that 
if  
   $(\{g_t :t\in \mathbb R\}, L/\Gamma)$ is ergodic then  
 given any $x\in L/\Gamma$ and any    $f\in C_c(L/\Gamma)$
\begin{equation}\label{eq;generic}
\lim_{T\to \infty}\frac{1}{T}\int_0^T f(g_tux)\dd t=\int_{L/\Gamma} f \dd \mu_{L/\Gamma}
\end{equation}
holds for almost every  $u\in L^+$ with respect to the Haar measure of $L^+$.
Suppose that $\mu$  is a  $\{g_t: t\in\mathbb R\}$ invariant  probability measure  on $L/\Gamma$.
We say that a   Radon measure  $\nu$ on a subgroup $U$ of $L$
  is  $(g_t, \mu)$ generic  at $x\in L/\Gamma$
if for any  $f\in C_c(L/
\Gamma)$ and   $\nu$ almost every $u\in U $ we have  (\ref{eq;generic}) holds.
We remark here that the property of being $(g_t, \mu)$ generic only depends on the equivalence class 
of the measure  $\nu$.

Unlike one parameter $\Ad$-unipotent subgroups few results are known about $\Ad$-diagonalizable subgroups when
 $\nu$ on $L^+$  is singular to the Haar measure. 
 We do know many examples of  probability measures $\nu$ whose pushforward  image 
 under $g_t$ as $t\to \infty$ or  trajectory under $\{g_t:0\le t\le T\}$ as $T\to\infty$
  is 
 equidistributed with respect to some  homogeneous probability measure.
The   reader can find precise descriptions of these  measures  for asymptotic results in Shah \cite{s96}\cite{s09}\cite{s092}\cite{s093}\cite{s10}, 
Shah and Weiss \cite{sw96}; and for average results by the  author and  Weiss in  
\cite{shi12}\cite{shi122}\cite{sw17}.

In this paper we  
 investigate pointwise equidistribution for measures studied in \cite{s96} and \cite{sw96} above. 
Let $G\le L$ be a connected semisimple Lie group  without compact factors.  Ratner's theorem \cite{r912} implies that 
for any $x\in L/\Gamma$ 
the orbit closure $\overline {Gx}$ is a finite volume homogeneous space, i.e. $\overline {Gx}=Hx$
 where $H=\{g\in L: g\overline {Gx}=\overline {Gx}\}$ and there is a unique $H$ invariant 
 probability measure (denoted by $\mu_{\overline{Gx}}$) supported on $\overline {Gx}$.
 Now we state the  
 main result of this paper. 
\begin{thm}\label{thm;1}
Let  $\{g_t :t\in \mathbb R\}$ be an  $\Ad$-diagonalizable 
one parameter subgroup of a connected semisimple Lie group  $G$. Let $G^+\le G$ be the unstable horospherical subgroup of $g_1$.
Suppose  that the projection of $g_1 $  to each simple factor of $G$ is not the identity element. 
Let $\Gamma$ be a lattice of a Lie group  $L$
which contains $G$. Then for every $x\in L/\Gamma$  the  Haar measure of $G^+$ is 
 $(g_t, \mu_{\overline {Gx}})$ generic at $x$.
\end{thm}

Our result is new in the following simple case: $G=
\left (\begin{array}{cc}
\SL_2(\mathbb R) & 0 \\
0 & 1
\end{array}
\right )$, $g_t=\mbox{diag}(e^t, e^{-t},1), 
L=\SL_3(\mathbb R), \Gamma=\SL_3(\mathbb Z)$.
The key property we use for the group $G^+$ is the $ g_1 $   expanding property which 
we describe now. 
Let $\{g_t :t\in \mathbb R\}$ and $G$ be as in Theorem \ref{thm;1}.
Every representation\footnote{In this paper a representation of $G$ means a 
	continuous map $\rho: G\to \GL(V)$ where $V$ is a
	nonzero finite dimensional  real  vector space. For $g\in G$ and $v\in V$, the
	linear action $\rho(g)v$ sometimes is denoted by $gv$ for simplicity.} $\rho$ of $G$ on a finite dimensional real  vector space $V$ splits
into a direct sum $V^+\oplus V^0\oplus V^-$  of $\rho(g_1)$ invariant subspaces so that  the 
restrictions of $\rho(g_1)$ to the spaces 
 $V^+, V^0, V^-$ have eigenvalues  $>,=,<1$ respectively. Let   $\pi_+$  be the
 $\rho(g_t)$ equivariant  projection from $V$ to $V^+$. 
A connected subgroup $U$ of $G$ normalized by $g_t\ (t\in \R)$ is said to be  $ g_1 $ expanding if
 for every nontrivial
irreducible  
representation
 $\rho$ of $G$ on $V$ and every nonzero vector $v\in V$ one has that the map 
\[
 U\to V \quad \mbox{given by}\quad  u\to  \pi_+(\rho(u)v)
\]
is not identically zero. 
It can be proved that    $U$ is $g_1$ expanding if and only if 
$U\cap G^+$ is $g_1$ expanding (see Lemma \ref{lem;characterize+}).
The existence of a $g_1$ expanding subgroup in the semisimple Lie group $G$ implies that $G$ has no compact factors, since otherwise $G$ has an irreducible representation 
$V=V^0$.  

One family of $ g_1  $ expanding subgroups comes from epimorphic subgroups of algebraic groups
introduced by  Bien and Borel \cite{bb}.
Suppose that $G$ is the connected component of real points of some
 semisimple linear algebraic group defined over $\mathbb R$.
Let  $S\le G$ be a one dimensional $\mathbb R$ split algebraic torus and let $U$ be a unipotent algebraic 
subgroup of $G$ normalized by $S$.  
Let   $H$  be the subgroup  generated by $S$  and $U$.
The group $H$ is  epimorphic in $G$ if 
any $H$ fixed  vector of 
 an   algebraic  representation of $G$
   is also  fixed by $G$. 
 It is proved in  \cite[Proposition 2.2]{sw96}  that if $H$ is an epimorphic subgroup of $G$ then $U$ is  $ g_1 $ expanding for some choice of the parameterization $\{g_t \}$ of the connected component of $S$.

 Under an additional abelian assumption for the  $g_1$ expanding  subgroup  $U$ we  prove the following refinement of Theorem \ref{thm;1}.
\begin{thm}\label{thm;2}
Let  $\{g_t :t\in \mathbb R\}$ be an  $\Ad$-diagonalizable
one parameter subgroup of a connected semisimple Lie group  $G$.
Let $\Gamma$ be a lattice of a Lie group  $L$ which
contains $G$.  
Suppose that  $U\le G^+$ is a connected  $ g_1 $ expanding abelian  subgroup of $G$.
Then for every  $x\in L/\Gamma$   the  Haar measure of $U$ is 
 $(g_t, \mu_{\overline {Gx}})$ generic at $x$.
\end{thm}

Here we give some  examples which are  motivations of  Theorem \ref{thm;2}. Let $m,n$
be two positive integers and let $\mathbf v=(a_1, \ldots, a_m, -b_1, \ldots,- b_n)$
where $a_i, b_j>0$ and $a_1+\cdots+ a_m=b_1+\cdots+b_n$. 
For every $\xi\in M_{mn}$ where  $ M_{mn}$ is the set   of $m\times n$
matrices, we let 
$u(\xi)=\left(
\begin{array}{cc}
I_m  & \xi\\
0    & I_n
\end{array}
\right)$
where $I_m$ and $I_n$ are identity matrices of order $m$ and $n$ respectively.
We consider the one parameter   diagonal subgroup  given by  
\[g_{t\mathbf v}=\mathrm{diag}(e^{a_1t}, \ldots, e^{a_mt}, e^{-b_1t},\ldots, e^{-b_nt}) \in   \SL_{m+n}(\R).
\] 
It follows from Kleinbock and Weiss \cite[Proposition 2.4]{kw08}  that the group
$U=\{u(\xi): \xi\in M_{mn}\}$ is $g_{\mathbf v}$ expanding. Therefore as a special case 
of Theorem \ref{thm;2} we have the following corollary.
\begin{cor}\label{cor;special}
Let $\Gamma$ be a lattice of $G=\SL_{m+n}( \mathbb R)$ and let $\mu$
be the probability Haar measure on $G/\Gamma$. Then for every $x\in G/\Gamma$ the 
additive Haar measure of $U=\{u(\xi): \xi\in M_{mn}\}$ is $(g_{t\mathbf v}, \mu)$
generic at $x$.
\end{cor}
Remark:
After this paper, in a joint work with Kleinbock and Weiss \cite{ksw} we  found a new proof of Corollary \ref{cor;special} while $\Gamma=\SL_{m+n}(\Z)$.  For compactly supported smooth functions, we also get a
	convergence  rate  of (\ref{eq;generic}). Using the same method, the author \cite{snew} proves  pointwise equidistribution with an error rate for 
	general $G$ and $\Gamma$ in the setting similar to  Corollary \ref{cor;special}. The method in \cite{ksw} and \cite{snew} does not apply to more general cases such as the example after 
	 Theorem \ref{thm;1}.

The abelian assumption of Theorem \ref{thm;2} for the group $U$ might be superfluous. 
The only place where we essentially need  it is the shadowing Lemma \ref{lem;shade} and its variant
in \S \ref{sec;five} which 
are links between random walks and flows. 
We do not know how to get shadowing lemma and  simultaneously the  contraction property
Lemma \ref{lem;estimate} 
 even in the case where $U$ is the two step  Heisenberg group and $g_t=g_{t\mathbf v}$ for $\mathbf v=(2,1, -3)$. 
 This is also the main obstruction that  we cannot apply our method 
 to  the case of volume measures of  curves studied in \cite{s09}\cite{s092}\cite{s093}\cite{s10}, 
e.g.~nonplanar analytic  curves
 in $G^+$ where $G=\SL(n, \mathbb R)$ and $g_t=\mbox{diag}(e^{(n-1)t}, e^{-t},\ldots, e^{-t})$. In a joint work with Fraczek and Ulcigrai \cite{new}, we prove 
 pointwise equidistribution for certain curves which are parameterized by a horospherical 
 subgroup.

Theorem \ref{thm;1} is deduced from Theorem \ref{thm;2} and the asymptotic  equidistribution of measures
proved in  \cite{sw96}. 
This type of deduction  might  be able to  prove pointwise equidistribution in some 
other cases where $U$ is not abelian.

The proof of Theorem \ref{thm;2} is based on
  quantitative estimate of the
$\{g_t: 0\le t\le T\}$ trajectory of  measures. The method is inspired by 
 Chaika and  Eskin \cite{ce} where they prove 
Birkhoff type ergodic theorem for Teichmuller geodesic flows on moduli spaces and 
 by Benoist and Quint \cite{bq132} where they prove almost everywhere equidistribution of Random walks on 
homogeneous spaces.

\textbf{Acknowledgements:}
The author would like to thank  Barak Weiss for suggesting  this problem and generously 
sharing his ideas. 
We also would like to thank Yves Benoist, Jean-Francois Quint, 
Alex Eskin and Alexander Gorodnik for discussions related to this work.

\section{Outline of the proof}\label{s;outline}

We first  outline the proof of Theorem \ref{thm;2} and leave details of the proof of Propositions 
\ref{prop;inv}, \ref{prop;nescape} and \ref{prop;singular} for later sections. 
Let $ G, g_t, U$ be as in Theorem \ref{thm;2}. 
In particular 
$U\le G^+$ is a connected abelian $g_1$ expanding  subgroup   of $G$. 
It follows from Ratner \cite[Proposition 1.3]{r90}  that 
$U$ is  simply 
 connected. We fix an isomorphism of Lie groups 
 \begin{equation}\label{eq;isomorphism}
 u:\mathbb R^m\to U
 \end{equation}
 so that there are positive real numbers $b_1, \ldots, b_m$ 
 such that  for  standard basis $\{\mathbf e_i\}_{1\le i\le m}$
of $\mathbb R^m $ one has 
\begin{equation}\label{eq;Recall}
g_tu(\mathbf e_i)g_{-t}=u(e^{tb_i}\mathbf e_i).
\end{equation}

It is not hard to see that  Theorem \ref{thm;2} is equivalent to the  following theorem which is more convenient to work on by our method.   
\begin{thm}\label{thm;proof}
Let  $\{g_t :t\in \mathbb R\}$ be an  $\Ad$-diagonalizable
one parameter subgroup of a connected semisimple Lie group  $G$.
Let $\Gamma$ be a lattice of a Lie group  $L$ which
contains $G$.  
Suppose that  $U\le G^+$ is a connected  $ g_1 $ expanding abelian  subgroup of $G$.
Let $x\in L/\Gamma$, let the interval $I=[-1, 1]$ and let $u$ be a fixed isomorphism as in (\ref{eq;isomorphism})
so that (\ref{eq;Recall}) holds. Then 
 for almost every $w\in I^m$ 
  \begin{equation}\label{eq;goalpoint}
\lim_{T\to \infty}  \frac{1}{T}\int_0^T g_tu(w)\delta_{x} \dd t=\mu_{\overline{Gx}}.
\end{equation}
\end{thm}

Here $g_t u(w)\delta_x$ is the pushforward of the Dirac measure $\delta_x$  by $g_tu(w)$ and 
it is equal to $\delta_{g_t u(w)x}$.
In the rest of this section we assume the notation and assumptions in Theorem \ref {thm;proof}.
The proposition below is  about  unipotent invariance.
\begin{prop}\label{prop;inv}
For almost every $w\in I^m$, if $\nu_w$ is any weak$^*$ limit of
$\frac{1}{T}\int_0^T g_t u(w) \delta_x\dd t$ as $T\to \infty$, then $\nu_w$ is invariant under $U$.
\end{prop}

Next we prove pointwise nonescape  of mass (Corollary \ref{cor;nescape})
using    quantitative nonescape of mass for $\{g_t: 0\le t\le T\}$ trajectory of the measure. 
For every measurable subset $K$ of $X$, positive real number $T$ and 
$w\in I^m$, we use $\mathcal A_K^T(w)$ to denote the proportion of the trajectory $\{g_tu(w)x: 0\le t\le T\}$ in $K$. 
More precisely,
\begin{equation}\label{eq;average}
\mathcal A_K^T(w):=\frac{1}{T}\int_0^T \mathbbm 1 _K(g_tu(w)x)\dd t
\end{equation}
where $\mathbbm 1_K$ is the characteristic function of $K$. 
For every measurable subset $J$ of $\mathbb R^m$ we let $|J|$ to denote 
the Lebesgue measure of $J$.
\begin{prop}\label{prop;nescape}
For every $0<\varepsilon <1 $, there is a compact subset 
 $K$ of $L/\Gamma$ and positive real  numbers  $a<1, C\ge 1$ such that 
 \begin{equation}\label{eq;pnes}
\left |\{w\in I^m: \mathcal A_K^T(w)\le 1- \varepsilon   \}\right |\le C  a^T
 \end{equation}
 for all $T> 0$.
\end{prop}

\begin{cor}\label{cor;nescape}
For almost every $w\in I^m$, any weak$^*$ limit point of 
$
\frac{1}{T}\int_0^T g_tu(w)  \delta_{x} \dd t
$
 as $T\to \infty $ is a probability measure.
\end{cor}
\begin{proof}
Given $0< \varepsilon <1$,  according to Proposition \ref{prop;nescape} there exists   a compact subset $K$ of $L/\Gamma $
 and  positive numbers  $a<1, C\ge 1$ so that  (\ref{eq;pnes}) holds
  as $T$ runs through all the positive integers.
So the
 Borel-Cantelli Lemma implies  that 
 \begin{equation}\label{eq;liminfeasy}
 \liminf_{n\to \infty, n\in \mathbb N}\mathcal A_K^n(w)\ge 1- \varepsilon
 \end{equation}
 for almost every  $w\in I^m$. If $w$ satisfies  (\ref{eq;liminfeasy})  then
 $
 \liminf_{T\to \infty} \mathcal A_{K}^T(w)\ge 1- \varepsilon . 
 $
 This means that any weak$^*$ limit $\nu_w$ of $
 \frac{1}{T}\int_0^T g_tu(w)  \delta_{x} \dd t
 $ satisfies $\nu_w(L/\Gamma) \ge 1- \varepsilon$. 
  The 
  conclusion follows by 
taking  $ \varepsilon $  arbitrarily  close to zero.
 \end{proof}

Let $H$ be the group generated by $\{g_t :t\in \mathbb R\} $ and $U$. It follows from 
Mozes \cite[Theorem 1]{m95}  that $H$  is an epimorphic subgroup of $G$.  
We say  a finite volume homogeneous subspace $Y$ of $L/\Gamma$ is $G$ ergodic if $G$ acts ergodically on $Y$ with respect to the probability homogeneous measure $\mu_Y$. 
Let $C_L(G)$ be the group of centralizers of $G$ in $L$.
It follows from  \cite[Proposition 2.1]{bq132}  that the set of  $G$ ergodic probability measures on $L/\Gamma$
 is at most a countable union of the set  
 \[
 C_L(G)\mu_Y:=\{g\mu_Y: g\in C_L(G)\}
 \]  where  $Y $ is a $G$ ergodic finite volume homogeneous subspace. 
 Without loss of generality we may assume that $\overline{Gx}=L/\Gamma$. 
We show that for almost every $w\in I^m$ any weak$^*$ limit $\nu_w$ 
 of 
$\frac{1}{T}\int_0^T g_tu(w)\delta_{x} \dd t$ as $T\to \infty $
 does not 
put any mass on 
$C_L(G)Y$ for any proper $G$ ergodic  finite volume homogeneous  subspace $Y$.
 This is  proved by  a similar quantitative result
for $\{g_t: 0\le t\le T\}$ trajectory of the measure.
\begin{prop}\label{prop;singular}
Suppose that $Gx$ is dense in $L/\Gamma$. 
Let $Y$ be a proper $G$ ergodic finite volume homogeneous subspace. 
For any  compact subset  $F$ of $C_L(G)$ and any $ \varepsilon >0$, there exists a 
compact subset $K$ of $L/\Gamma$
with $K\cap FY=\emptyset$
 and positive numbers  $a<1, C\ge 1$ such that 
 \begin{equation*}
\left |\{w\in I^m: \mathcal A_{K}^T(w)\le 1- \varepsilon   \}\right |\le C  a^T
 \end{equation*}
 for all $T> 0$.
\end{prop}
\begin{cor}\label{cor;nobad}
Suppose that $Gx$ is dense in $L/\Gamma$.  
Let $Y$ be a proper $G$ ergodic finite volume homogeneous subspace. 
Then for almost every $w\in I^m$ one has $\nu_w(C_L(G)Y)=0$ for 
any
 weak$^*$ limit $\nu_w$ of $\frac{1}{T}\int_0^T g_tu(w)\delta_x\dd t$
 as $T\to \infty $.
\end{cor}
The proof  uses Proposition \ref{prop;singular} and  is the same as that of Corollary \ref{cor;nescape}, so we 
omit the details here. 

\begin{proof}[Proof of Theorem \ref{thm;proof}]
It follows from  Ratner's  orbit closure  theorem \cite{r912} that  $Gx$ 
is dense in a $G$ ergodic finite volume homogenous subspace of $L/\Gamma$.
So we can without loss of generality assume that 
$Gx$ is dense in $L/\Gamma$. 
  
It follows from Proposition \ref{prop;inv}, Corollaries  \ref{cor;nescape}  and 
\ref{cor;nobad} that there exists a  subset  $J$ of $I^m$ with full measure such that  for
any $w\in J$,  any weak$^*$ limit $\nu_w$ of 
$\frac{1}{T}\int_0^T g_tu(w)\delta_x\dd t$ 
  as $T\to \infty$ has the following properties: 
\begin{itemize}
\item $\nu_w$ is invariant under $U$, and moreover invariant under $G$.
\item $\nu_w$ is a probability measure.
\item $\nu_w(C_L(G)Y)=0$ for any proper $G$ ergodic  finite volume homogeneous subspace $Y$.
\end{itemize}
Therefore for any $w\in J$ we have (\ref{eq;goalpoint})
holds. This completes the proof.
\end{proof}

Here  we describe a general strategy of using Theorem \ref{thm;2} to prove 
pointwise equidistribution for other $g_1$ expanding subgroups not necessarily abelian.  In particular we derive 
Theorem \ref{thm;1} from it.  We need to use the following result about asymptotic 
equidistribution of measures.  
\begin{thm}[\cite{s96},\cite{sw96}]\label{thm;old}
Let $U'$ be a connected $\Ad$-unipotent $ g_1 $ expanding subgroup of $G$. Suppose that $\mu$ is a
probability measure on $U'$ absolutely continuous with respect to the Haar measure. 
Let $\mu_ x$ be the push forward of $\mu$ to $L/\Gamma$
with respect to the map $u\in U'\to ux$.
Then 
\[
g_t \mu_ x\to \mu_{\overline{Gx}} \quad \mbox{ as } t \to \infty.
\] 
\end{thm}
This result is not explicitly stated in both of the papers, but  it can be derived 
 easily from the main results there as we will see. 
Let $H$ be the subgroup of $G$ generated by $\{g_t: t\in \mathbb R\}$ and $U'$.
Then the  $g_1$ expanding property implies that the Zariski closure of $\rho(H)$ is 
an epimorphic subgroup of $\rho(G)$ for any finite dimensional real representation $\rho$ of 
$G$. Furthermore the ray $\{g_t: t>0\}$ is contained in the cone of 
\cite[Lemma 2.1]{sw96}  for any
nontrivial 
irreducible representation $\rho$. Therefore Theorem \ref{thm;old} follows from \cite[Theorem 1.4]{sw96}.

Given $x\in L/\Gamma$, if 
$
\lim\limits_{T\to \infty}\frac{1}{T}\int_0^T g_tu \delta_{x} \dd t
$
exists in the space of probability measures on $L/\Gamma$ for almost every $u\in U'$, then 
 Theorem \ref{thm;old} implies that  the limit has to be $\mu_{\overline {Gx}}$
 almost surely. 
 Hence the Haar measure of $U'$ is $(g_t, \mu_{\overline{Gx}})$ generic at $x$. 
The almost surely existence of the limit can be obtained by 
Theorem \ref{thm;2} 
 if there is an abelian subgroup  subgroup $U_a$ of $U'$ 
normalized by $g_t$ and a connected semisimple subgroup $G_1$ of $G$ without compact factors 
so that the following holds:
\begin{itemize}
\item[$(*)$]
$\{g_t: t\in \mathbb R\}$ is a subgroup of
$G_1$ and 
$U_a$ is a $ g_1 $ expanding subgroup of $G_1$.  
\end{itemize}

For example, if $G$ is the real rank one group $SU(2,1)$ then
the unstable horospherical  subgroup $G^+$ is not abelian. 
But we can take $G_1$ to be a subgroup whose Lie algebra is  isomorphic to  $\mathfrak{sl}_2$.  The next lemma shows that this strategy always works   when $U'=G^+$. 
\begin{lem}\label{lem;abelian}
Under the assumption of Theorem \ref{thm;1} there is a connected
 semisimple  subgroup $G_1\subset G$ without compact factors 
and an abelian subgroup $U_a$ of $G^+$ such that property $(*)$ holds.
\end{lem}
The proof of this lemma uses  strongly orthogonal system of simple 
root systems
and will be given in the appendix.  
Lemma \ref{lem;abelian} together with 
Theorem \ref{thm;2} and  Theorem \ref{thm;old}  proves
Theorem \ref{thm;1}.

\section{Some auxiliary results}
\subsection{Large deviation}
In this section we prove a large deviation result. 
Our argument is inspired by 
\cite{bq13} and \cite{a}.

Let $(W, \mathcal B, \mu)$ be a standard Borel  space with probability measure $\mu$. 
The  conditional expectation of 
a  nonnegative   Random variable 
$\xi$ (i.e.~a measurable  map $\xi: W\to [0, \infty]$)\footnote{Although we allow 
	functions take the value $\infty$, we always assume they are not  $\infty$ 
	almost surely.}
with respect to   
  a sub sigma algebra
$\mathcal F$ of $\mathcal B$     is an $\mathcal F$ measurable function 
$E(\xi| \mathcal F)$  such that for any $A\in \mathcal F$ one has
$\int_A\xi(w)\dd \mu(w)=\int_AE(\xi| \mathcal F)(w)\dd \mu(w)$.
The conditional probability of $A\in \mathcal B$  is  the function
$
\mu(A|\mathcal F):=E(\mathbbm 1_A|\mathcal F)
$ 
where $ \mathbbm 1 _A$ is the characteristic function of $A$. 
For a nonnegative  random variable $\xi$ and $a\in \mathbb R$
we will follow the convention of probability theory to write 
$\mu(\xi \ge a)$ for $\mu(\{w\in W: \xi(w)\ge a\})$ and $E(\xi)$
for $\int _W \xi (w) \dd w$.

In this paper the set of natural numbers  $\mathbb N=\{0, 1, 2, \ldots\}$.
A measurable map $\xi: W \to { \mathbb N}\cup \{\infty\}$
  is called  ${\mathbb N}$ valued random variable. 
A sequence of random variables $(\xi_i)_{i\in \mathbb N}$ is said to be  
increasing if $\xi_i(w)\ge \xi_{i-1}(w)$ for all $i\ge 1$ and almost every $w\in W$. 
A sequence of   sub sigma algebras  $(\mathcal F_i)_{i\in \mathbb N}$
of  $\mathcal B$ is said to 
be a filtration if 
$\mathcal F_{i-1}\subseteq \mathcal F_{i}$.
In the rest of  this section the relations $=$ and $\le $ for functions on $W$ are meant
to hold almost surely.

\begin{lem}\label{lem;main}
Let $(\xi_i)_{i\in \mathbb N}$ be an  increasing sequence  of
${\mathbb N}$ valued  Random variables  on $W$.  Let $(\mathcal F_i)_{i\in \mathbb N}$ be a sequence of 
filtrations of sub sigma algebras of $\mathcal B$ such that $\xi_i$ is $\mathcal F_i$ measurable. 
Suppose that there exits $\vartheta_0>0$ and $C_0 \ge 1$ such that  
\begin{equation}\label{eq;basic exp}
\mu(\xi_{i}-\xi_{i-1}\ge  q|\mathcal F_{i-1})\le C_0 e^{-q\vartheta_0 }
\end{equation}
for all  $q,i\ge 1$.
Then  for any $\varepsilon>0$
there exist positive numbers   $Q $ and $\vartheta$ such that for 
every  positive integer  $n$  we have
\begin{equation}\label{eq;goaldev}
\mu\left( \frac{1}{n}\sum_{i=1}^n \mathbbm 1 _Q ( \xi_{i}(w)-\xi_{i-1}(w))\ge \varepsilon\right)
\le e^{-\vartheta n}
\end{equation}
where 
$ \mathbbm 1 _Q: {\mathbb N}\to  {\mathbb N} $ is defined by \begin{equation}\label{eq;truncation}
 \mathbbm 1 _Q(q)=\left\{
\begin{array}{cl}
q & \mbox{ if } q\ge Q \\
0 & \mbox{otherwise.}
\end{array}
\right.
\end{equation}
\end{lem}
Remark:
It can be seen from the proof below that $Q$ and $\vartheta$ only depend on $C_0,\varepsilon$ and $\vartheta_0$; but they do  not depend  
on 
the probability space, 
the sequence of random variables or the filtration of sigma algebras. 

\begin{proof}
We will show that  (\ref{eq;goaldev}) holds for   
\begin{equation}\label{eq;choiceQ}
\vartheta=\frac{\varepsilon\vartheta_0}{4}\quad \mbox{and}\quad Q\ge 
\frac{2\log[C_0^{-1}(e^{\varepsilon\vartheta_0/4}-1)(1-e^{-\vartheta_0/2})]}
{-\vartheta_0}.
\end{equation}

For every positive integer  $n$ we define a function $f_n$ on
$W $ by
 \[
 f_n(w)=\exp\left (\frac{\vartheta_0}{2}{\sum_{i=1}^n ( \mathbbm 1 _Q  ( \xi_{i}(w)-\xi_{i-1}(w))} \right).
 \]
 Since $f_{n-1}$ is $\mathcal F_{n-1}$ measurable, we have 
 $$E(f_n|\mathcal F_{n-1})
 =f_{n-1}E\left (\exp\Big(\frac{\vartheta_0}{2} \mathbbm 1 _Q  ( \xi_{n}(w)-\xi_{n-1}(w))\Big)|\mathcal F_{n-1}\right).$$
So by the  monotone convergence theorem and (\ref{eq;basic exp}),  we have
 \begin{eqnarray*}
E(f_n|\mathcal F_{n-1}) & \le &  f_{n-1}\left [1+
\mathcal \sum_{q\ge Q} e^{\vartheta_0 q/2}
\mu (\xi_n-\xi_{n-1}=q|\mathcal F_{n-1})
\right] \\
 & \le & f_{n-1}\frac{1-e^{-\vartheta_0/2}+C_0e^{-Q\vartheta_0/2}}{1-e^{-\vartheta_0/2}}.
 \end{eqnarray*}
An induction on $n$ gives
 \[
E(f_n)\le \left (\frac{1-e^{-\vartheta_0/2}+C_0e^{-Q\vartheta_0/2}}{1-e^{-\vartheta_0/2}}
\right )^n.
 \]
On the other hand by Chebyshev inequality 
\[
E(f_n)\ge
 e^{\varepsilon n\vartheta_0/2}\mu \left(\frac{1}{n}\sum_{i=1}^n \mathbbm 1 _Q  ( \xi_{n}-\xi_{n-1})\ge\varepsilon
\right).
\]
Therefore 
\begin{equation}\label{eq;derivation}
\mu\left(\frac{1}{n}\sum_{i=1}^n \mathbbm 1 _Q  ( \xi_{n}-\xi_{n-1})\ge \varepsilon
\right)
\le \left(\frac{1-e^{-\vartheta_0/2}+C_0e^{-Q\vartheta_0/2}}
{e^{\varepsilon \vartheta_0/2}(1-e^{-\vartheta_0/2})}
\right)^n.
\end{equation}
So (\ref{eq;goaldev}) follows from   (\ref{eq;choiceQ}) 
and (\ref{eq;derivation}).
\end{proof}

\subsection{unipotent invariance}
The aim of this section is to prove Proposition \ref{prop;inv}. 
Our argument is modeled on  \cite[\S 3]{ce}. 
Let $L, \Gamma, G, g_t, U,x$ be   as  in Theorem \ref{thm;proof}.

We observe that there exists a countable
dense subset of $C_c(L/\Gamma)$  consisting of  smooth functions.
   Also if $s_1, s_2\in \R$ are linearly independent over 
$\mathbb Q$, then the closure of the group generated by  $u(s_1\mathbf e_j), u(s_2\mathbf e_j)$ $(1\le j\le m)$ is $U$. 
Therefore, Proposition \ref{prop;inv} will follow if  we can show that given any 
 $\psi \in C_c^\infty(L/\Gamma), s>0$ and $1\le i\le m$
we have for almost every $w\in I^m$
\begin{equation}\label{eq: goal}
\lim_{T\to \infty}\frac{1}{T}\int_0^T \psi_t(w) \dd w\to 0
\end{equation}
where 
\[
\psi_t(w)  =  \psi
\big (g_tu(w)x\big)-\psi\big(u(s\mathbf e_i)g_tu(w)x\big).
\]

We will prove (\ref{eq: goal}) using the   law of large numbers. The key is the the 
following effective  estimate of correlations. 
\begin{lem}\label{lem;decay}
There exists $\vartheta>0$ and $C\ge 1$ such that for any $t,l>0$
\begin{equation}\label{eq: cross}
\left |\int_{I^m}\psi_t(w)\psi_l(w)\dd w  \right| \le C e^{-\vartheta |l-t|}.
\end{equation}
\end{lem}

Lemma \ref{lem;decay} allows us to use the following lemma  to complete the proof of 
 (\ref{eq: goal}) and hence Proposition \ref{prop;inv}.

\begin{lem}[\cite{ce} Lemma 3.4]
Suppose that $\psi_t:I^m\to \mathbb R$ are bounded functions satisfying (\ref{eq: cross}) (for some $C\ge 1$ and $\vartheta>0$).  Additionally,  assume that
there exists $C_1\ge 1$ such that  $\psi_t(w)$ are $C_1$-Lipschitz functions of $t$ for each $w\in I^m$. Then
 (\ref{eq: goal}) holds
 for almost 
every $w\in I^m$.
\end{lem}

\begin{proof}[Proof of Lemma \ref{lem;decay}]

 We fix a right invariant Riemannian metric on $L$
and let $d(\cdot, \cdot)$ be the induced  distance function.
We note that the function  $\psi $ is Lipschitz, 
i.e.~$|\psi(gy)-\psi(hy)|\ll d(g, h) $ for any $g, h\in L$ and $y\in L/\Gamma$. 
So there exists $C_1\ge 1$, such that $\psi_t(w)\ (w\in W)$ are $C_1$-Lipschitz functions of $t$.  

Without loss of generality we assume that $l>t$ and $i=1$.
 Let 
 $b=b_1>0$ which is defined in the beginning of \S \ref{s;outline}, i.e. 
$
 {g_1}u(\mathbf e_1) g_1^{-1}=u(e^b\mathbf e_1).
$
Then 
\[
\psi_t(w)=\psi\big(g_tu(w)x\big)-\psi\big ( g_tu(w+se^{-bt }\mathbf e_1)x\big ).
\]
We will show that for $\vartheta=b/2$ there exists $C\ge 1$ so that (\ref{eq: cross}) holds.

 We divide $[-1, 1]$ consecutively  into intervals of 
the form 
\[
I(r)=[r-e^{-(l+t)b/2}, r+e^{-(l+t)b/2}]
\]
except for the last part which will not affect the validity of (\ref{eq: cross})
since it has length  
less than $2e^{-(l+t)b/2}$.
For every $s_1\in \mathbb R$ with $|s_1|\le e^{-(l+t)b/2}$
we have
 \[
d(g_tu(s_1 \mathbf e_1),g_t)=d(u(e^{bt}s_1\mathbf e_1), 1_L)\ll e^{-(l-t)b/2}.
\]
As noted above the function $\psi$ is Lipschitz,
so for every 
$s_1\in I(r)$ and 
$w\in \{0\}\times I^{m-1}$
 one has
\[
|\psi_t(s_1\mathbf e_1+w)-\psi_t(r\mathbf e_1+w)|\ll e^{-(l-t)b/2}.
\]
Therefore for any $w\in \{0\}\times I^{m-1}$
 \begin{eqnarray}\label{eq;inv1}
  & &\frac{1}{|I(r)|}\int_{I(r)} \psi_l(s_1\mathbf e_1+w)\psi_t(s_1\mathbf e_1+w)\, ds_1  \\
 & =&\frac{\psi_t(r\mathbf e_1+w)}{|I(r)|}
\int_{I(r)}\psi_l (s_1\mathbf e_1+w)\, ds_1 + O(e^{-(l-t)b/2}).\notag
\end{eqnarray}
Since the interval $I(r)$ and $I(r)+se^{-bl}$ have overlaps except for their ends whose   length 
are $se^{-bl}$,
we have 
\begin{equation}\label{eq;inv2}
\frac{1}{|I(r)|}
\left|\int_{I(r)}\psi_l (s_1\mathbf e_1+w)\, ds_1  \right| \ll 2se^{-(l-t)b/2}.
\end{equation}
For any $w\in \{0\}\times I^{m-1}$,
\begin{align}\label{eq;toolong}
\int_{I}\psi_t(s_1\mathbf e_1+w)\psi_l(s_1\mathbf e_1+w)\dd s_1=\sum_{I(r)} \int _{I(r)} 
\psi_t\psi_l\dd s_1+O(e^{-(l-t)b/2}),
\end{align}
where the sum is  over a covering of $[-1,1]$ by consecutive intervals of the form $I(r)$ except for the ends.  Then 
(\ref{eq;inv1}), (\ref{eq;inv2}) and (\ref{eq;toolong}) imply   
\[ 
\left|\int_I \psi_t(s_1\mathbf e_1+w)\psi_l(s_1\mathbf e_1+w)\dd s_1\right|
\ll e^{-(l-t)b/2}
\]
for all $w\in \{0\}\times I^{m-1}$. 
So  (\ref{eq: cross}) follows from this estimate and the
 Fubini theorem. 
\end{proof}

\subsection{Linear representations}
Let $ G, g_t, U$ be   as  in Theorem \ref{thm;proof}.
 The main result of this section is the following lemma about uniform expanding property.

 \begin{lem}\label{lem;estimate}
Let $V$  be a finite dimensional  representation of 
$G$ with norm $\|\cdot\|$. Suppose that    $V$ has no nonzero  $G$-fixed vectors.  
 Then there exist  positive real numbers 
 $\lambda_0,\vartheta_0$ with  the following properties:
  for every $0<\vartheta<\vartheta_0$ there exits $T_\vartheta>0$ such that   
if 
  $\kappa: I^m\to \mathbb R_{\ge 0}$ is a  measurable function    with
 $\inf_{w\in I^m} \kappa(w)\ge t\ge T_\vartheta$ one has 
  \begin{equation}\label{eq;goal estimate}
\sup _{ \|v\|=1}\int_{I^m} \frac{\dd w}{\|g_{\kappa(w)} u(w)v\|^\vartheta}\le  e^{-\lambda_0\vartheta  t}.
 \end{equation}
\end{lem}

Recall that   
 $V=V^+\oplus V^0\oplus V^-$ be the decomposition  according to the  eigenvalues
of $g_1$  and  $ \pi_+: V\to V^+$ be the  projection map. 
 For every $v\in V, r>0$
we set 
\[
D^+(v, r)=\{w\in I^m: \|\pi_+(u(w)v)\|\le r\}.
\]

\begin{lem}\label{lem;good}
Let $V$ be as in Lemma \ref{lem;estimate}.
Then there exists $\vartheta_0>0$ such that
\begin{equation}\label{eq;uniform}
C:=\sup _{ \|v\|=1, r>0}\frac{|D^+(v,r)|}{r^{\vartheta_0}}<\infty.
\end{equation}
\end{lem}
\begin{proof}
Recall that   $U$ is  assumed to be $ g_1 $ expanding in $G$. 
When $v$ varies in the unit sphere of $V$ the family of  maps  which send
$w\in I^m$ to $ \pi_+(u(w)v)$ are polynomials in $w$ with  degree uniformly bounded from above and
maximum of absolute values of  coefficients uniformly bounded  from below
by some positive constant.  So
the lemma follows from the $(C, \alpha )$-good property of polynomial functions  in 
\cite[Lemma 3.2]{bkm01} .
\end{proof}

\begin{proof}[Proof of Lemma \ref{lem;estimate}]
The  proof here  is the  same as   \cite[Lemma 5.1]{emm98}. 
We take  $\vartheta_0>0$ so that (\ref{eq;uniform}) holds. 
Let $b>0$ so that 
$e^{b}$ is the smallest eigenvalue of $g_1$ in $V^+$.
 We will show that the lemma holds    for this $\vartheta_0$ and $\lambda_0=\frac{b}{2}$. 

As different norms on the    finite dimensional vector space $V$ are equivalent,  there exists $C_1\ge 1$ such that for every vector $v\in V$ and $t\ge 0$ one has
$e^{bt} \|\pi_+(v)\| \le C_1\|g_t v\|$. 
Let $C$
 be the constant in (\ref{eq;uniform}) and let
 \[
 r=\sup_{\|v\|=1, w\in I^m} \|\pi_+(u(w)v)\|.
 \]
 Given a positive real number $\vartheta<\vartheta_0$, we choose    $T_\vartheta>0$ such that 
 \begin{equation}\label{eq;esti aug1}
\frac{2^{\vartheta_0}CC_1^{\vartheta_0}r^{\vartheta_0-\vartheta}}{1-2^{\vartheta-\vartheta_0}}e^{-b\vartheta T_\vartheta/2} \le 1. 
 \end{equation}
We show that (\ref{eq;goal estimate}) holds for $\kappa$ with $\inf \kappa \ge t\ge T_\vartheta$.

We fix a unit vector $v\in V$ and estimate the integral of 
\[
{f_{\kappa,v}(w)=\|g_{\kappa(w)}u(w)v\|^{-\vartheta}}.
\]
Since  
  $
\|g_{\kappa(w)}u(w)v\|\ge C_1^{-1} e^{bt}\| \pi_+(u(w)v)\|
$
for all $w\in I^m$, 
one has   
\begin{equation}\label{eq;simplified}
f_{\kappa,v}(w)\le C_1^{\vartheta_0}e^{-bt\vartheta}\| \pi_+(u(w)v)\|^{-\vartheta}.
\end{equation}
For every nonnegative integer $n$, 
 (\ref{eq;uniform}) and (\ref{eq;simplified}) imply 
\begin{equation}\label{eq;sieve}
\int\limits_{ D^+(v,2^{-n}r)\setminus D^+(v,2^{-n-1}r)}f_{\kappa,v}(w) \dd w\le e^{-bt\vartheta}2^{\vartheta_0}CC_1^{\vartheta_0}r^{\vartheta_0-\vartheta}2^{-n(\vartheta_0-\vartheta)}.
\end{equation}
We write 
\[
I^m=D^+(v, 0)\cup\Big(\cup_{n\ge 0}\big( D^+(v,2^{-n}r)\setminus
 D^+(v, 2^{-n-1}r)\big)\Big).
\]
Since $|D^+(v, 0)|=0$, we have
\begin{eqnarray*}
\int\limits_{I^m}f_{\kappa,v}(w) \dd w  & = &\sum_{n=0}^\infty\int\limits_
{ D^+(v,2^{-n}r)\setminus  D^+(v, 2^{-n-1}r)}f_{\kappa, v}(w) \dd w \\
\mbox{by (\ref{eq;sieve})}\qquad&\le & 
\frac{2^{\vartheta_0}CC_1^{\vartheta_0}r^{\vartheta_0-\vartheta}}{1-2^{\vartheta-\vartheta_0}}e^{-bt\vartheta} \\
\mbox{by (\ref{eq;esti aug1})}\qquad & \le  & e^{-\lambda_0\vartheta  t}.
\end{eqnarray*}
\end{proof}

\section{Nonescape of mass}
The aim of this section is to prove Proposition \ref{prop;nescape}. 
Let $L, \Gamma, G, g_t, U,x$ be  as  in Theorem \ref{thm;proof} and let  $X=L/\Gamma$.
The main tool is the contraction property of a function $\alpha$ 
on $X$ (we call it height function) which measures whether  points in  $X$  
 are close to  $\infty$. 
 Height functions with the contraction property  on homogeneous spaces  
are introduced by 
Eskin, Margulis and Mozes \cite{emm98}.
A significant   improvement  is given by Benoist and
 Quint \cite{bq12} which  will be used in this paper. 

\subsection{Existence of height function}

\begin{lem}\label{lem;contract}
	There exist positive real numbers $\lambda_0, T_0$ such that 
for any compact subset  $Z$ of $ X$ and $t\ge T_0$
 there exists a lower semicontinuous   function  $\alpha: X\to [0, \infty]$  and $b>0$ with the following properties:
\begin{enumerate}[\rm (1)]
\item
For  every $y\in X$
\begin{equation}\label{eq;linear}
\int _{I^m}\alpha(g_t u(w)y) \dd w \le e^{-t\lambda_0 } \alpha (y)+b.
\end{equation}

\item $\alpha $ is finite  on $ G Z$.

\item $\alpha$ is Lipschitz, i.e.~for every compact subset $F$
of $G$ there exists $C\ge 1$ such that $\alpha(gy)\le C \alpha(y)$ for every $y\in X$ and $g\in F$.

\item $\alpha$ is proper, i.e.~if $\alpha(Z_0)$ is bounded for some subset $Z_0$ of $X$ then
$Z_0$ is relatively compact.  
\end{enumerate}
\end{lem}
Remark: Here lower semicontinuity implies that for every positive number $M$ the subset  $\alpha^{-1}([0, M])$ is  closed and  hence compact by (4).

We first deal  with the case where  $\Gamma$ is arithmetic.
For the moment we assume that   $L=\SL_d(\mathbb R)$, $\Gamma =\SL_d(\mathbb Z)$  $(d\ge 2)$ and  review the height function defined in \cite{bq12}.  It is well known that the space 
$X=\SL_d(\mathbb R)/\SL_d(\mathbb Z)$ can be identified with the set
of unimodular lattices in $\mathbb R^d$. For every $y\in X$, let $\Lambda_y$ be
the lattice in $\mathbb R^d$ corresponding to it, i.e.~$\Lambda_y=g\mathbb Z^d$ 
if $y=g\SL_d(\mathbb Z)$.
A vector 
\[
v\in \wedge^* \mathbb R^d
:=\oplus_{0\le i\le d} \wedge^i \mathbb R^d
\]
 is monomial if $v=v_1\wedge \cdots \wedge v_i$ 
where $v_1, \ldots, v_i\in \mathbb R^d$. We say $v$ is $y$-integral monomial   if
we can take $v_i\in \Lambda_y$.

Recall that $L=\SL_d(\mathbb R)$ has a natural structure of real algebraic groups.
Since $G\le L$ is a 
connected  semisimple Lie group,
it is the connected component of  an algebraic subgroup.  
We fix a maximal connected diagonalizable  subgroup 
$A$ of $G$ containing $\{g_t: t\in \mathbb R\}$ and normalizing $U$. Let $\Phi(G, A) $ be the 
 relative root system, i.e.~ the set of 
nonzero  weights   of $A$
appeared in the adjoint representation.
We fix a positive system $\Phi(G, A) ^+$ such that $\eta (g_1)\ge 1$ for 
every $\eta\in \Phi(G, A) ^+$.  
We endow a partial order on the set $P$ of algebraic characters  of  $A$ by 
$
\eta \le \mu$  
if and only if 
$ \mu-\eta$
 is  nonnegative linear combination  of  $\Phi(G, A) ^+$.
 For any   irreducible 
representation of $G$,
the set of weights of $A$ in this representation 
 has a unique maximal element called
 highest weight of the representation. 
Let $P^+$ be the set of all the highest weights appearing in $\wedge^*\mathbb R^d$.

 For each $\eta\in P^+$, let $\pi_\eta$ be the projection 
from $\wedge^* \mathbb R^d$ to the subspace consisting of all the irreducible sub representations with
highest weight $\eta$. 
Let  $\|\cdot\|$ be the usual Euclidean norm on $\wedge^*\mathbb R^d $.
One of the key ingredients of \cite{bq12} is the following Mother Inequality.
\begin{lem}[\cite{bq12} Proposition 3.1]\label{lem;mo}
There exists $C_1\ge 1$ such that for any monomials $u, v, w$ in $\wedge ^*\mathbb R^d$
one has the inequality 
\[
\|\pi_\eta(u)\|\cdot \|\pi_\mu (u\wedge v\wedge w)\|\le C_1 \max_{\nu, \rho\in P^+ \atop 
\nu+\rho\ge \eta +\mu} \|\pi_\nu(u\wedge v)\|\cdot
\|\pi_\rho(u\wedge w)\|.
\]
\end{lem}


We fix the following index:
\[
\delta_i=(d-i)i \quad\mbox{and}\quad \delta_\eta= \log( \eta( g_1))
\]
for  $0\le i\le d$ and $\eta\in P^+$. Recall that
$U$ is $ g_1 $ expanding, so for
 $\eta \in P^+\setminus 0$ (where $0$ is the trivial  character of $A$)  we have
$\delta_\eta>0$.
 Also we take
\begin{equation}\label{eq;kappa}
\sigma=(\min_{\eta\in P^+\setminus 0} \delta_\eta)^{-1}
\quad\mbox{and}\quad \sigma_1=(\max_{\eta\in P^+\setminus 0} \delta_\eta)^{-1}.
\end{equation}

Let $\varepsilon>0$ and $0<i< d$. 
For every  
$v\in \wedge^i \mathbb R^d$  we let
\[
\varphi_{ \varepsilon }(v)=\left\{
\begin{array}{ll}
\min _{\eta\in P^+\setminus 0}  \varepsilon ^{\frac{\delta _i}
{\delta _\eta}}\|\pi_\eta(v)\|^{\frac{-1}{\delta _\eta}} & \mbox{ if }
\|\pi_0(v)\|<\varepsilon ^{\delta _i}\\
0 & \mbox{otherwise.}
\end{array}
\right.
\]
We remark here 
that  $\varphi_{ \varepsilon }(v)=\infty $
 if
 $v=\pi_0(v)$ and $\|v\|<  \varepsilon ^{\delta _i}$. 

\begin{lem}\label{lem;restate}
There exist positive numbers $\lambda,\vartheta, T$ with the following properties:
for all
$t\ge T$,
  $v\in \wedge ^i\mathbb R^d$  $(0<i<d)$ and $0<\varepsilon <1$ one has
\begin{equation}\label{eq;swim}
\int_{I^m} \varphi_{ \varepsilon }^\vartheta(g_tu(w)v) \dd w\le e^{-\lambda \vartheta t }\varphi_{ \varepsilon }^\vartheta
(v)  .
\end{equation}
\end{lem}

\begin{proof}
Let $V$ be the   $G$ invariant subspace in $\wedge^* \mathbb R^d$ complementary to     $\pi_0(\wedge^* \mathbb R^d)$.
For the representation of  $G$ on $V$ we fix $\vartheta_0, \lambda_0>0$ so that   Lemma \ref{lem;estimate} holds.
Then for every  $\frac{\vartheta_0}{2\sigma\delta_\eta}$ where $\eta\in P^+\setminus 0$, there exists $T_\eta>0$ such that
$(\ref{eq;goal estimate})$ holds  for $t\ge T_\eta$. 
We show that the lemma holds for  $\vartheta = \vartheta_0/2\sigma$, $\lambda=\lambda_0\sigma_1$ and $T=\max_{\eta\in P^+\setminus 0} T_\eta$.

Assume that  $v\neq \pi_0(v)$ and $\|\pi_0(v)|< \varepsilon ^{\delta_i}$. 
By 
Lemma \ref{lem;estimate}   for every $\eta\in P^+\setminus 0$ with $\pi_\eta(v)\neq 0$ and $t\ge T\ge T_\eta$
\[
\int_{I^m}\|\pi_\eta(g_tu(w)v)\|^{\frac{-\vartheta}{\delta_\eta}}\dd w
\le e^{-\lambda_0 \vartheta t/\delta_\eta} \|\pi_\eta(v)\|^{\frac{-\vartheta}{\delta_\eta}}
\le e^{-\lambda \vartheta t } \|\pi_\eta(v)\|^{\frac{-\vartheta}{\delta_\eta}}.
\]
Therefore (\ref{eq;swim}) holds.

 If either 
 $\|\pi_0(v)\|\ge \varepsilon^{\delta _i}$ or    $\pi_0(v)=v$ and $\|v\|< \varepsilon^{\delta _i} $,
  then both sides of (\ref{eq;swim}) are either $0$ or  $\infty$ respectively.
So (\ref{eq;swim}) holds trivially.
\end{proof}

Following \cite{bq12} we define  $\alpha_{\varepsilon }: X\to [0, \infty]$ by 
\[
\alpha_{ \varepsilon }(y)=\max\varphi_{\varepsilon }(v)
\]
where the maximum is taken over all the non-zero $y$-integral monomials $v\in \wedge ^i\mathbb R^d$
with $0<i<d$.

\begin{lem}\label{lem;inequality}
	Let $\vartheta, \lambda$ and $T$ be given as in Lemma \ref{lem;restate}
	and let $T_0=T+\frac{2\log (2d)}{\lambda \vartheta}$. 
Then   for any $t\ge T_0$ there exists 
 $\varepsilon ,
b>0$ such that for every $y\in X$
\begin{equation}\label{eq;swim5}
\int_{I^m} \alpha_{ \varepsilon }^{\vartheta}(g_tu(w)y) \dd w\le  e^{-\lambda\vartheta t /2}\alpha_{ \varepsilon }^\vartheta(y)+b.
\end{equation}
\end{lem}
\begin{proof}
We fix $t\ge  T_0$ and set 
 $C_0=\sup\{\|g_tu(w)\|+\|(g_tu(w))^{-1}\|: w\in I^m\}\ge 1$
where $\|\cdot\|$ is the operator norm for elements of $G$ acting on $\wedge^*\mathbb R^d$.
We take $\varepsilon $ small enough so that 
\begin{align}\label{eq;detail}
C_0^{2\sigma}(C_1\varepsilon)^{\sigma_1/2} <1
\end{align}
where $C_1$ is the constant given in Lemma \ref{lem;mo} and $\sigma,\sigma_1$
are defined in (\ref{eq;kappa}).
Let 
\[
b_{1}=\sup \varphi_{\varepsilon }(v)<\infty
\]
where the supremum is taken over all the monomials $v\in\wedge^i \mathbb R^d
\ (0< i <d)$
with $\|v\|\ge 1$. We will show that (\ref{eq;swim5}) holds for 
\[
b=2^m(C_0^{\sigma}\max\{b_1,C_0^{2\sigma}\})^\vartheta.
\]

 It follows from the definition of $C_0$ that  for every monomial
 $v\in \wedge^i\mathbb R^d$ with $0< i <d$ one has 
\[
C_0^{-\sigma}\varphi_{\varepsilon}(v)\le \varphi_{\varepsilon}(g_tu(w)v)\le C_0^{\sigma} \varphi_{\varepsilon}(v).
\]
If $\alpha_{ \varepsilon }(y)\le \max\{b_1, C_0^{2\sigma}\}$, then
\[
\int_{I^m} \alpha_{ \varepsilon }^\vartheta(g_tu(w)y) \dd w\le b.
\]

Let $\Psi$ be the finite set of primitive $y$-integral and monomial elements $v$ 
of $\wedge^*\mathbb R^d$ with degrees in $(0,d)$ such that 
\[
\varphi_{ \varepsilon }(v)\ge C_0^{-2\sigma} \alpha_{\varepsilon}(y).
\]
Then for all $w\in I^m$,
\[
\alpha_{\varepsilon}(g_tu(w)y)=\max_{v\in \Psi} \varphi_{\varepsilon}(g_tu(w)v).
\]
It follows from
claim (5.9) 
 in the  proof of \cite[Proposition 5.9]{bq12}\footnote{The claim is proved 
 	using (\ref{eq;detail}) and some corollaries of Lemma \ref{lem;mo}.}   that
 if  $\alpha_{ \varepsilon }(y)>\max\{b_1, C_0^{2\sigma}\}$ 
 then 
 $\Psi$ contains at most one element up to sign change in each degree $i$. 
Therefore,
in this case  Lemma \ref{lem;restate} implies 
\begin{align*}
\int_{I^m} \alpha_{ \varepsilon }^\vartheta(g_tu(w)y) \dd w&\le \sum _{v\in \Psi}
\int_{I^m} \varphi_{ \varepsilon }^\vartheta(g_tu(w)v) \dd w \\
&\le 
{e^{-\lambda\vartheta t}}\sum _{v\in \Psi} \varphi_{ \varepsilon }^\vartheta(v)\\
&\le {e^{-\lambda\vartheta t}}  2d \cdot \alpha_{ \varepsilon }^\vartheta(y)\\
&\le e^{-\lambda\vartheta t/2}  \alpha_{ \varepsilon }^\vartheta(y). 
\end{align*}
\end{proof}

\begin{lem}
	\label{lem;arithmetic}
Suppose that  $X=L/\Gamma=\SL_d(\R)/\SL_d(\Z)$. Then Lemma \ref{lem;contract} holds. 
\end{lem}
\begin{proof}
	We fix $\lambda, T_0, \vartheta$ as in Lemma \ref{lem;inequality}. 
	We show that Lemma \ref{lem;contract} holds for $\lambda_0=\lambda \vartheta/2$
	and $T_0$.
	Given a  compact subset $Z$ of $X$, by Mahler's compactness criterion there
	exists $\varepsilon >0$ such that $\alpha_\varepsilon^\vartheta$ is finite on $Z$.
	For $t\ge T_0$, by possibly making $\varepsilon$ smaller, Lemma \ref{lem;inequality} implies that there exist $b>0$ such that (\ref{eq;swim5}) holds. Therefore, (1) and (2) of Lemma \ref{lem;contract} hold for $\alpha=\alpha^\vartheta_\varepsilon$ .
	The lower semicontinuity and Lipschitz property (3) can also be checked directly 
	from the definitions. The property (4) follows from Mahler's compactness criterion. 
\end{proof}

 The general case of Lemma \ref{lem;contract} will 
be reduced to the  arithmetic case 
and the  rank one case below. 
\begin{lem}
	\label{lem;rank one}
	Suppose that  $L$ is a connected semisimple Lie group with   (real) rank one. Then 
	Lemma \ref{lem;contract} holds. 
\end{lem}
\begin{proof}[Sketch of Proof]	
	If $X=L/\Gamma$ is compact, then we can simply take $\alpha(y)=1$ for all $y\in X$.
	Suppose that  $X$ is noncompact. 
	It follows from \cite{gr70} (cf.~\cite[Proposition 3.1]{kw} 
	and \cite[Page 54]{bq12}) and the proof of  \cite[Proposition 2.7]{em04}  that there 
	exists a finite dimensional representation $V$ of $G$ with norm $\|\cdot\|$ and nonzero 
	vectors  $v_1, \ldots, v_r$ of $V$
	with the following properties: 
	
	\begin{enumerate}[(a)]
		\item $\Gamma v_i$ is closed and hence discrete in $V$ for $1\le i\le r$.
		\item For any $F\subset L$, the set $F\Gamma\subset L/\Gamma$ is relatively compact if and only if 
		there exists $a>0$ such that $\|g\gamma v_i\|>a$ for any $\gamma\in \Gamma, g\in F$
		and $1\le i\le r$.
		\item There exists $a_0>0$ such that for 
		any $g\in L$ there exists at most one $v\in \bigcup_{1\le i \le r} \Gamma v_i$  such that $\|gv\|<a_0$.
	\end{enumerate}
	Let
	\[
	\tilde \alpha_\vartheta(g\Gamma)=\max_{1\le i\le r}\max_{\gamma\in \Gamma}\|g\gamma v_i\|^{-\vartheta}.
	\]
	Lemma 
	\ref{lem;contract}
	follows from  properties (a)-(c) listed above and 
	Lemma \ref{lem;estimate} (for the action of $G$ on the maximal $G$ invariant subspace of $V$ having no nonzero $G$ fixed vectors) by taking $\alpha=\tilde \alpha_\vartheta$ for some  $\vartheta $ sufficiently small.
\end{proof}

We also need  need the following 
lemma which is straightforward to check. 
\begin{lem}\label{lem;straight}
Let $\Gamma_1$ be a lattice of a connected Lie group $L_1$. Let $\varphi: L\to L_1$ be a surjective
homomorphism of Lie groups so that $\varphi(G)$ is nontrivial. Suppose that $\varphi(\Gamma)\subset 
\Gamma_1$ and the induced map $X=L/\Gamma\to L_1/\Gamma_1$ is proper. If Lemma \ref{lem;contract}
holds for  $L_1/\Gamma_1, \varphi(g_t), \varphi(U)$ or it holds for 
$L/\Gamma', g_t, U$ where $\Gamma'$ is a finite index subgroup of $\Gamma$, then it holds for $X, g_t, U$. 
\end{lem}

\begin{proof}
	[Sketch of Proof] Let $\alpha_1$ and $\alpha'$ be height functions on $L_1/\Gamma_1$  and $L/\Gamma'$ respectively so that (1)-(4) of Lemma 
	\ref{lem;contract} hold. Then in the first case we can take 
	$\alpha=\alpha_1\circ \varphi$; and in the latter case we can take 
	\[
	\alpha(g \Gamma)=\sum_{i=1}^q \alpha(g \gamma_i \Gamma') 
	\]
	where $\gamma_1 ,\ldots, \gamma_q$ is a  complete list of representatives of
	the left cosets of 
	 $\Gamma/\Gamma' $. 
\end{proof}

\begin{proof}[Proof of Lemma \ref{lem;contract}]
Let $\mathfrak {r}$ be the largest amenable ideal of the Lie algebra $\mathfrak l$  of $L$, $\mathfrak s:=\mathfrak l/\mathfrak r$,
$S:=Aut(\mathfrak s)$.
Let $R$ be the kernel of the  adjoint representation $\Ad_{\mathfrak s}: L\to S$.
 It follows from  \cite[Lemma 6.1]{bq12}   that 
$\Gamma\cap R$ is a cocompact lattice  in $R$ and the image group $\Gamma_S:=\Ad_{\mathfrak s}(\Gamma)$
is a lattice in $S$. Therefore the map
$
L/\Gamma\to S/\Gamma_S
$
is proper. According to Lemma \ref{lem;straight}  it suffices to prove the case where $L$ is a connected semisimple center free Lie group without compact 
factors. 

Under this assumption we can write $L=\prod_{i=1}^qL_i$ as a direct product  of connected semisimple
Lie groups such that $L_i\cap \Gamma$ is an irreducible lattice in $L_i$.
We can 
 assume that $\Gamma=\prod_{i=1}^q L_i\cap \Gamma$ since the latter 
 has finite index in $\Gamma$. 
 Let $\pi_i: L\to L_i$ be the natural quotient map.
 We also use $\pi_i$ to denote the induced map $L/\Gamma\to L_i/\Gamma_i$ according 
 to the context.  
If $\pi_i(G)$ is nontrivial then  $\pi_i(g_t)$ is a nontrivial  $\Ad$-diagonalizable  one parameter subgroup 
of $\pi_i(G)$
and $\pi_i(U)$ is  $\pi_i(g_1)$ expanding.
Suppose that  Lemma \ref{lem;contract} holds for every $L_i/\Gamma_i$ with $\pi_i(G)$ nontrivial. Assume without loss of generality that $\pi_i(G)$ is nontrivial for $1\le i\le p$ and $\pi_i(G)$ is trivial for $p<i\le q$.
Then we can find $\lambda_i, T_i>0$ $(1\le i\le p)$ such that 
for every $t\ge T_0:=\max \{T_i: 1\le i\le p \}$ and the  compact subset $\pi_i(Z)\subset L_i/\Gamma_i\ (1\le i\le p)$ there exists a lower semicontinuous
function 
 $\alpha_i: L_i/\Gamma_i\to [0, \infty]$ satisfying (1)-(4) of Lemma \ref{lem;contract}.
 If $\pi_i(G)$ is trivial, we set $\alpha_i=(1-\mathbbm 1_{\pi_i(Z)})\cdot \infty$.
Then the function $\alpha$ on $X$ defined by
\[
\alpha(y_1, \ldots, y_q)= \alpha_1(y_1)+\cdots+\alpha_q(y_q)\quad
\mbox{where} \quad y_i\in L_i/\Gamma_i
\]
satisfies properties (1)-(4) of Lemma \ref{lem;contract} with respect to $Z,t$
and $\lambda=\min\{\lambda_i: 1\le i\le p \}$.
Therefore it suffices  to prove  the case where $L$ is a connected  center free
semisimple Lie group without compact factors  and $\Gamma$ is an irreducible lattice. 

 If the (real) rank of $L$ is bigger than or equal to two, then Margulis arithmeticity theorem (see e.g. \cite[Theorem 6.1.2]{z})
implies that there is an injective map \[
\varphi: L\to \SL_d(\mathbb R)
\]
such that $\varphi(\Gamma)$ is commensurable with $\varphi(L)\cap \SL_d(\mathbb Z)$.
So Lemma \ref{lem;contract}   follows from Lemmas \ref{lem;arithmetic} and \ref{lem;straight}.
If the rank of $L$ is one, then we can apply Lemma \ref{lem;rank one}
to complete the proof.



\end{proof}

\subsection{Exponential recurrence to cusp}\label{s;rec}
Let $\lambda_0,T_0>0$ be as in Lemma \ref{lem;contract}.
For $Z=\{x\}$ and  $t\ge T_0$ we choose a height function  $\alpha:X\to [0, \infty]$
and $b>0$ so that (1)-(4) in Lemma \ref{lem;contract} hold.\footnote{One should consider  $t$ and $\alpha$
	as fixed in this and the next section. But in the next section we will endow an additional  condition on the lower bound of it.} 
We first  use inequality (\ref{eq;linear}) to study discrete  trajectories
\begin{equation}\label{eq;distrajectory}
\{g_{nt}u(w)x:  n\in \mathbb N\}\quad (w\in I^m).
\end{equation}
Recall that $\{\mathbf e_i: 1\le i\le m\}$ is the standard basis of $\mathbb R^m$ and $b_i>0\ (1\le i\le m)$  satisfy 
(\ref{eq;Recall}). Let 
\[
w=\sum_{i=1}^m a_i\mathbf e_i\quad \mbox{and }\quad w'=\sum_{i=1}^m a'_i \mathbf e_i.
\]
If $|a_i-a'_i|\le  2e^{-ntb_i}$, then two points
 $g_{nt}u(w)x$ and $g_{nt}u(w')x$
 can always be translated to each other by elements in the   compact subset 
$u([-2, 2]^m)$ of $G$.
In view of property (3) of $\alpha$ we consider them as at the same height. 
The following lemma plays a key role to link random walks with respect to $g_tu(I^m)$ and the 
 trajectories  (\ref{eq;distrajectory}).

\begin{lem}[Shadowing Lemma]\label{lem;shade}
For $1\le i\le m$ let $J_i\subset [-1, 1]$ be an interval with length  $ |J_i|\ge e^{-ntb_i}$.
Then for any nonnegative measurable  function $\psi$ on $X$ and $J=\prod_{i=1}^m J_i$ one has
\begin{equation}\label{eq;3 years}
\int_J \psi(g_{(n+1)t}u(w)x) \dd w \le \int_J\int_{I^m} \psi\big(g_{t}u(w_1)g_{nt}u(w)x\big)  \dd w_1\dd w.
\end{equation}
\end{lem}
\noindent
 \begin{proof}
 The proof is the same for $m=1$ and arbitrary $m$, since $U$ is assumed to be abelian. For simplicity we only give details in the case where 
 $m=1$.
 Since $g_{t}u(s_1) g_{nt}=g_{(n+1)t}u(s_1e ^{-ntb_1})$,
  the right   hand side of (\ref{eq;3 years}) is equal to 
 \begin{align*}
 \int_J\int_{I} \psi\big(g_{(n+1)t}u(s+s_1e^{-ntb_1})x\big)  \dd s_1\dd s.
 \end{align*}	
 After   making change of variables $(s_1,s)\to (s_1, \tilde s)=(s_1, s+ s_1 e^{-ntb_1})$,
   we have  the above integral is
\begin{align*}
	\ge \int_{J}\int_{I(\tilde s)}\psi\big( g_{(n+1)t} u(\tilde s) x  \big)\dd s_1\dd \tilde s
\end{align*}
where $I{(\tilde s)}=\{s_1\in I: \tilde s-s_1 e^{-ntb_1}\in J \}$. 
Since $|J|\ge e^{-nt b_1}$ and $I=[-1, 1]$, one has $|I{(\tilde s)}|\ge 1$. 
Therefore (\ref{eq;3 years}) holds. 
\end{proof}
In Lemma \ref{lem;shade} the abelian assumption is essential to us. If we drop the abelian  assumption, then we need
to change the domain of the integral for $w_1$ to something that depends on $J$. In that case it is 
not clear to the author how to get (\ref{eq;goal estimate})
uniformly  for various domains determined by $J$
 and hence the contraction property
(\ref{eq;linear}).

For every positive integer $n$
we need to 
 divide $I^m$ into boxes of sides $e^{-ntb_i}\ (1\le i\le m)$
 so that the above shadowing lemma holds and all the $g_{nt}u(w)x$ for $w$
in a  box are not far away to each other.
We can do this consecutively in each component  $I$ except for the last interval  which we allow
 to have length bigger than $e^{-ntb_i}$ but 
no more than $2 e^{-ntb_i}$. 
We want  the
 partition for   $n+1$  to be a refinement of that  for  $n$, so 
we do this by induction on $n$. 
In the first step we divide $I^m$ into boxes of the form
\[
\prod_{1\le i\le m}[-1 +je^{-tb_i}, -1+(j+1)e^{-tb_i})
\]
with slight modifications for the end  intervals. 
For every $w\in I^m$ we use $I_1(w)$ to denote  the box containing $w$. 
In the second step we  divide each box above into  smaller boxes of the form
\[
\prod_{1\le i\le m}[-1 +je^{-tb_i}+ke^{-2tb_i}, -1+je^{-tb_i}+(k+1)e^{-2tb_i})
\]
and we use
 $I_2(w)$  to denote the box containing $w$. 
By the same construction we do it for all $n$ and define $I_n(w)$ accordingly.
We also take $I_0(w)=I^m$
 for every $w\in I^m$.
Note that for all  $n\in \mathbb N$,   $w\in I^m$ and $w'\in I_n(w)$ one has
\begin{equation}\label{eq;compactG}
g_{nt}u(w')x=hg_{nt}u(w)x\quad \mbox{for some}\quad h\in u([-2,2]^m).
\end{equation}

For every $n\in \mathbb N$ let $\mathcal B_n$ be the smallest sigma algebra of $I^m$ generated by  $I_j(w)$  $(0\le j\le n,w\in I^m)$. 
It is not hard to see that the atom of $w$ in $\mathcal B_n $ is $I_n(w)$
and the sequence $(\mathcal B_n)_{ n\in \mathbb N}$ is a filtration of sigma algebras.

\begin{lem}\label{lem;basic ineq}
For every  $J\in \mathcal B_n\  (n\in \mathbb N)$  one has
\[
\int_{J}\alpha(g_{{(n+1)}t}u(w)x) \dd w\le e^{-t\lambda_0 }\int_{J}\alpha(g_{nt}u(w)x) \dd w +b|J|.
\]
\end{lem}
\begin{proof}
The lemma follows from  shadowing Lemma \ref{lem;shade} and  the linear inequality (\ref{eq;linear}).
\end{proof}

Let us fix a positive real number $l_0$ so that 
\begin{align}
\label{eq;l0large}  &\frac{b}{l_0}+e^{-t\lambda_0}\le e^{-t\lambda_0 /2}, \\
 \label{eq;xinl0} & x\in X_{l_0}\mbox{ where }X_{l_0}=\{y\in X: \alpha(y)\le l_0\}.
\end{align}
We 
 define a sequence of measurable  functions $\sigma_i: I^m\to \mathbb N\cup\{\infty\}$ which represents  the
$i$th 
return time to the compact subset $X_{l_0}$.
To begin with we set $\sigma_0(w)=0$.
To apply the  shadowing lemma we  want $\{w\in I^m: \sigma_i(w)=n\}$ to be 
$\mathcal B_n$ measurable. 
 The formal definition 
is
\begin{equation}\label{eq;return time}
\sigma_i(w)=\inf\{ n>\sigma_{i-1}(w): g_{nt}u(w_1)x\in X_{l_0} \mbox{ for some } w_1\in I_n(w)\}.
\end{equation}
We take the convention that $\inf \emptyset=\infty$. 
In particular, if   $\sigma_{i-1}(w)=\infty $  then  $\sigma_i(w)=\infty$.
For  simplicity we set
\begin{align}\label{eq;Isigma}
I(\sigma_n,w)=I_{\sigma_n(w)}(w).
\end{align}

\begin{lem}\label{lem;pre expon}
There exists   $C_0\ge 1$  such that for all  $q, n\in \mathbb N$ with $q\ge 1$ and $w_0\in I^m$
with
$\sigma_n(w_0)<\infty$
the measure of the  set 
\begin{equation}\label{eq;expjump}
J_{n,q}(w_0)=\{w\in I(\sigma_n, w_0): \sigma_{n+1}(w)-\sigma_n(w)\ge q\}
\end{equation}
is at most 
$ C_0e^{-q t \lambda_0/2}|I(\sigma_n, w_0)|$.
\end{lem}
Remark: It follows from Lemma \ref{lem;pre expon} that  $\sigma_n(w)<\infty$ almost surely for all $n\in \mathbb N$. 
\begin{proof}
We fix $w_0$, $n$ and write  $\sigma_n=\sigma_n(w_0), J_q={J_{n,q}}(w_0)$ for simplicity.
Let 
\[
\quad s_q=\int_{J_{q+1}}\alpha(g_{(\sigma_n+q)t}u(w)x) \dd w \qquad (q\in \N)
.
\]
Note that if $w\in J_{q+1}\ (q\ge 1)$, then $\alpha(g_{\sigma_n+q}u(w)x)>l_0$. So 
by Chebyshev inequality 
\begin{align}\label{eq;modify1}
s_{q}> l_0 |J_{q+1}|\quad
\mbox{for all } q\ge 1.
\end{align}

Since $J_{q}\ (q\ge 1)$ is $\mathcal B_{\sigma_n+q-1}$ measurable,  Lemma \ref{lem;basic ineq}  implies
\begin{align}\label{eq;modify}
s_q\le  \int_{J_{q}}\alpha(g_{(\sigma_n+q)t}u(w)x) \dd w\le e^{-t\lambda_0 } s_{q-1}+b|J_{q}|.
\end{align}
By (\ref{eq;l0large}) and (\ref{eq;modify1}) (\ref{eq;modify}) , 
$$
s_q\le \left(e^{-t\lambda_0}+\frac{b}{l_0}\right)s_{q-1}\le e^{-t\lambda_0/2}s_{q-1}
\qquad (q\ge 2).$$
An  induction on $q$ gives
\begin{align}
\label{eq;modify2}
s_q\le s_1 e^{-(q-1)t\lambda_0 /2}\le e^{-qt\lambda_0/2}s_0+b e^{-(q-1)t\lambda_0/2}|J_1| \qquad (q\ge 1),
\end{align}
where in the last inequality we use (\ref{eq;modify}).

Since $J_1=I(\sigma_n, w_0)$, one has 
$
s_0=\int_{I(\sigma_n, w_0)}\alpha(g_{\sigma_n t}u(w)x)\dd w.
$
 It follows from the definitions that there exists 
$w'\in I(\sigma_n, w_0)$ such that $\alpha(g_{\sigma_n t}u(w')x)\le l_0$. 
By property (3) of $\alpha$ in Lemma \ref{lem;contract}, there exists $C\ge 1$ depending on the compact subset $u([-2,2]^m)$ in (\ref{eq;compactG}) such that 
$\alpha (g_{\sigma_n t}u(w)x)\le Cl_0$ for all $w\in I(\sigma_n, w_0)$.
Using  Chebyshev inequality  again for $s_0$ one has 
\begin{align}
\label{eq;details}
s_0\le |I(\sigma_n, w_0)| C l_0.
\end{align}
 Therefore (\ref{eq;l0large}), (\ref{eq;modify1}), (\ref{eq;modify2}) and (\ref{eq;details})  imply  
 for all $q\ge 1$ 
  $$|J_{q+1}|\le (C+e^{t\lambda_0/2}b/l_0) e^{-qt\lambda_0 /2 }
  |I(\sigma_n, w_0)|\le 2C e^{-qt\lambda_0 /2 }
  |I(\sigma_n, w_0)|
 .
  $$
So (\ref{eq;expjump}) holds for $C_0=2Ce^{t \lambda_0 /2} $.
\end{proof}

Recall that the proportion of the trajectory $\{g_tu(w)x: 0\le t\le T\}$ in a subset $K$ of $X$ 
is defined in (\ref{eq;average}).  
A discrete version of this function  is 
defined by
\begin{equation}\label{eq;daverage}
\mathcal D_{K}^n(w)= \frac{1}{n}\sum_{i=0}^{n-1} \mathbbm 1_{K}(g_{it}u(w)x)
\end{equation}
where $n$ is a positive integer and $\mathbbm 1_K$ is the characteristic function of $K$.

 Let $\mathcal C_n$ be the smallest sigma algebra of $I^m$ generated by   $I(\sigma_i, w)$  for
 $0\le i\le n$ and $w\in I^m$ with $\sigma_i(w)<\infty$.
Note that modulo null sets every element  of $\mathcal C_n$ is a disjoint  union of at most countably many   sets of the form $I(\sigma_n, w)$ with $\sigma_n(w)<\infty$. 

\begin{lem}\label{lem;discrete}
For every $0<\varepsilon_0<1$ 
there exists a compact subset $K_0$ of $X$ and $0<a_0<1$ such  that
\begin{equation}\label{eq;gotoUK}
\left| \left\{
w\in I^m:\mathcal D_{K_0}^n(w)\le 1-\varepsilon_0\}
\right\}
\right|\le 2^m a_0^n
\end{equation}
for all positive integer $n$.
\end{lem}
\begin{proof}
Recall that  $l_0>0$ is fixed  so that   (\ref{eq;l0large}) and (\ref{eq;xinl0})  hold.
By Lemma \ref{lem;pre expon}  there exists $C_0\ge 1$ and $\vartheta_0=\frac{t\lambda_0}{2}>0$ such that   the measure of  the set $J_{n, q}(w)$ decays exponentially for all $n\in \N$ and $q\ge 1$.
Using Lemma \ref{lem;main} (with $W=I^m,\mu=\frac{1}{2^m}\mathrm{Leb}, \xi_n=\sigma_n, \mathcal F_n=\mathcal C_n$), it follows  that there exists $Q\ge 1$
and $ 0< a_0<1$   such that  for every positive integer $n$ the measure of the  set 
\[
J_n=\left\{w\in I^m:\frac{1}{n} \sum_{i=1}^n  \mathbbm 1 _Q (\sigma_i(w)-\sigma_{i-1}(w))\ge \varepsilon_0
\right \}
\]
is at most  $2^m  a_0^n$. Here  $ \mathbbm 1 _Q$ is the truncation of the  identity function    defined in (\ref{eq;truncation}).

The exponential decay of the measure of $J_n$  is very close to (\ref{eq;gotoUK}). We will prove 
(\ref{eq;gotoUK}) by enlarging $X_{l_0}$ to a bigger compact subset $K_0$.
We claim that the lemma holds for 
\begin{align} \label{eq;tech}
 K_0= \bigcup_{s\in [0,Qt] }g_{s }u([-2, 2]^m)X_{l_0}.
\end{align}
It suffices to prove that for any positive integer  $n$
\begin{align}\label{eq;sunday}
 \{
w\in I^m:\mathcal D_{K_0}^n(w)\le 1-\varepsilon_0\}\subset J_n.
\end{align}

We fix $w$ in the left hand side of (\ref{eq;sunday}).
Let  $0<i_1< \cdots<  i_k<n$ be the sequence of consecutive times $i$ for which $g_{it}u(w)\not\in K_0$.
Since $\mathcal D_{K_0}^n(w)\le 1-\varepsilon_0$, we have $k/n\ge \varepsilon_0$.
To prove $w\in J_n$ it suffices to find
a subset $R$ of $\{1,2,  \ldots,n \}$ such that
\begin{enumerate}[(1)]
\item $\sigma_r(w)-\sigma_{r-1}(w)\ge  Q$  for every  $r\in R$;
\item 
for every $1\le j\le k$  there exists 
$r\in R$  such that  $\sigma_{r-1}(w)<i_j< \sigma_r(w)$.
\end{enumerate}
This amounts  to say that each consecutive block of $\{0\le i<n: g_{it}u(w) x\not \in K_0  \}$ is contained in some interval of the form $[\sigma_{r-1}(w), \sigma_r(w)]$
which has length at least $Q$. This is not difficult to believe since $K_0$ is constructed from $X_{l_0}$ by (\ref{eq;tech}).
The  proof  we give  is technical and  has some inductive flavor. 

Recall that $x\in X_{l_0}$ and  $\sigma_0(w)=0$. For the first step, we let 
\begin{align*}
 m_1& =\max \{ i<i_1: i=\sigma_r(w) \mbox{ for  some } r\ge 0\}\ge 0, \\
 m'_1& =\min \{ i>i_1: i =\sigma_r(w)\mbox{ for  some } r\ge 0\}\le \infty.
\end{align*}
Then there exists a positive integer $r=r_1$ with   $r\le i_1$ such that
  $m_1=\sigma_{r-1}(w)$ and  $m_1'=\sigma_{r}(w)$.
It follows from the definitions that 
\begin{align*}
y:=g_{m_1t}u(w)x&\in u([-2, 2]^m)X_{l_0}\subset K_0, \\
g_{it}y&\in K_0 \quad\mbox{for} \quad i\le   Q, \\
g_{(i_1-m_1)t }y&\not\in K_0.
\end{align*}
Therefore,  $$\sigma_r(w)-\sigma_{r-1}(w)=m_1'-m_1\ge i_1-m_1\ge Q$$ which verifies (1).

If  $i_k< m_1'$ then $R=\{r_1\}$ also satisfies (2) and
we are done. Otherwise we choose the smallest $j$ such that
$ i_{j}>m_1'$.
Then we can repeat the construction to find $r=r_2$ with $r_1<r\le i_{j}$ so that
$\sigma_{r-1}(w)<i_{j}<\sigma_{r}(w)$ and (1) holds.
We continue this procedure until
for $r=r_s$ we have
$i_k<\sigma_{r}(w)$.
It follows directly from the construction that $R=\{r_1, \dots, r_s\}$ satisfies (1)
and (2).
\end{proof}

The following lemma allows us to deduce the continuous version of exponential 
recurrence  from 
the discrete version in Lemma \ref{lem;discrete}. It will also be used in the next 
section.

\begin{lem}\label{lem;continuousd}
Let $ \varepsilon_0<\frac{1}{2},  a_0<1$ be positive numbers  and let ${K_0}$ be a compact subset of $X$. Suppose $x\in K_0$ and  
(\ref{eq;gotoUK}) holds for every positive integer $n$.
Then there exist positive numbers  $ a<1, C\ge 1$ and a compact subset $ K\subset GK_0$ such that  for all $T>0$
\begin{equation}\label{eq;booktrain}
\left |\{w\in I^m: \mathcal A_{ K}^T(w)\le 1- 2\varepsilon_0   \}\right |\le  C a^T.
\end{equation}
\end{lem}
\begin{proof}
Let  $T_0=\frac{2}{\varepsilon_0}t$.
We show that (\ref{eq;booktrain}) holds    for 
$$  K=\bigcup_{s\in [0,T_0]}g_su(I^m)K_0, \quad  a=a_0^{\frac{1}{t}}\quad  \mbox{ and }\quad C=a_0^{-1}2^m.$$

 Since $x\in K_0$,  (\ref{eq;booktrain}) holds trivially for $T\le T_0$.
Now we assume   $T> T_0$. 
 We claim that  if $w\in I^m $ satisfies
  $\mathcal A_{ K}^T(w)\le 1-2\varepsilon_0$
   then $\mathcal D_{K_0}^{\lfloor T/t\rfloor}(w)\le 1-{\varepsilon_0}$ 
 where ${\lfloor T/t\rfloor} $ is the biggest  integer less  than or equal to $T/t$.
   Given $i\in \mathbb N$,
   if
   $g_{it}u(w)x\in  K_0 $ then $g_su(w)x\in K$ for $s\in [it, (i+1)t]$.
 Suppose  $\mathcal D_{K_0}^{\lfloor T/t\rfloor}(w)> 1-{\varepsilon_0}$, then 
\[
\mathcal A_{ K}^T(w)> \frac{\lfloor T/t\rfloor(1-\varepsilon_0) t-t}{T}\ge 
 \frac{ T(1-\varepsilon_0) -2t}{T}\ge 1-\varepsilon_0 -\frac{2t}{T_0}=1-2\varepsilon_0.
\]
This completes the proof of the claim. 
So by (\ref{eq;gotoUK}),  the left hand side of (\ref{eq;booktrain}) is
\[
\le 2^ma_0^{\lfloor T/t\rfloor}\le 2^m a_0^{-1}a_0^{ T/t}=Ca^T.
\]
\end{proof}

\begin{proof}[Proof of Proposition \ref{prop;nescape}]
It follows from  Lemmas \ref{lem;discrete} and   \ref{lem;continuousd}.
\end{proof}

\section{Exponential recurrence to  singular subspace}\label{sec;five}
The aim of this section is to prove Proposition \ref{prop;singular}. 
Let $L, \Gamma, G, g_t$, $U,x, Y , C_L(G), F$ be
 as in Proposition \ref{prop;singular}. We write  $X=L/\Gamma$ and  $S=\{g\in L: gY=Y\}$.  Let
 $\mathfrak s, \mathfrak c,\mathfrak l$ be the Lie algebras of $S,C_L(G), L$, respectively.
Let $\mathfrak t$ be  a $ G $ invariant subspace of $\mathfrak l$ complementary to 
$\mathfrak s+\mathfrak c$ with respect to the adjoint action.

We fix a norm $\|\cdot\|$ on $\mathfrak g$ and use 
$\|g\|$ to denote the   operator norm of $g\in G$ with respect to
the adjoint representation. 
There exits positive numbers $\vartheta', T'$ such that
\begin{equation}\label{eq;expgrow}
\max{( \|g_{t}u(w)\|, \|(g_{t}u(w))^{-1}\|)}\le e^{t\vartheta' }
\end{equation}
for all $w\in I^m$ and   $t\ge T'$. 
We fix positive numbers  $\vartheta_0, \lambda_0, T_0>0$ with the following properties:
\begin{itemize}
	\item For  $\vartheta_0$ and  $\lambda_0$  Lemma \ref{lem;estimate} holds
	with respect to the adjoint action  of $G$ on $\mathfrak t$.
	\item  Lemma \ref{lem;contract} holds  for   $\lambda_0$ and  $T_0$.   
\end{itemize}

We fix a positive real number $\vartheta<\vartheta_0$ 
 sufficiently small so that $\frac{\lambda_0 }{2}-\vartheta'\vartheta>0$. 
By Lemma \ref{lem;estimate}, there exists $T_\vartheta>0$ 
such that for any measurable map $\kappa: I^m \to
{ \mathbb N}\cup\{\infty\}$ with $\inf_{w\in I^m} \kappa(w) \ge t\ge T_\vartheta$ 
\begin{equation}\label{eq;lcon}
\sup_{\|v\|=1, v\in \mathfrak t}\int_{I^m}\frac{\dd w}{\|g_{\kappa(w)}u(w)v\|^{\vartheta}}\le e^{-t\vartheta\lambda_0}.
\end{equation}

We fix $t\ge \max \{  T_0, T_\vartheta , T'\}$ and choose a height function $\alpha
: X\to [0, \infty]$ so that (1)-(4) in Lemma \ref{lem;contract} holds with respect to 
$Z=\{x\}$  and some $b>0$. 
For every $w\in I^m$ and $n\in \mathbb N$ we let $I_n(w)$ be the box defined in \S \ref{s;rec}. 
We fix $l_0>0$ so that (\ref{eq;l0large})
and (\ref{eq;xinl0}) hold.  Let    $\sigma_i: I^m\to {\mathbb N}\cup \{\infty\}$
be the $i$th return time to $X_{l_0}$  defined in (\ref{eq;return time}).
By  Lemma \ref{lem;pre expon}
 there exists  $C_0\ge 1$  such that for all $w\in I^m, n\in \N$ and $q\ge 1$
 \begin{align}
 \label{eq;lhard}
  |J_{n,q}(w)|\le C_0e^{-qt\lambda_0/2}|I(\sigma_n,w)|
 \end{align}
 where  $J_{n,q}(w)$ is
 defined in (\ref{eq;expjump}) and $I(\sigma_n, w)$ is defined in (\ref{eq;Isigma}).

Now we  define 
a height function $\beta:X\to [0, \infty]$  
which roughly speaking measures whether 
elements of a fixed compact subset  are close to $FY$. 
The construction here is the same as  \cite[\S 6.8]{bq13}. 
Let $N\supset u([-2, 2]^m)$ be a relatively   compact open  neighborhood of the identity  in $G$. The role of $N$ is to guarantee    $g_{\sigma_n(w)t}u(I({\sigma_n},w))\subset N X_{l_0}$
(for all $n\in \N$ and $w\in I^m$ with $\sigma_n(w)<\infty$)
 and the lower semicontinuity of 
$\beta$ below.  
 We choose a positive number $\varepsilon\le 1 $, an open neighborhood $O$ of identity in $C_L(G)$
 and finite number of elements 
$f_1, \ldots, f_k\in F$ with $F\subset Of_1\cup \cdots \cup O f_k$
  so that the following holds:  for any $y\in NX_{l_0}$ and any $f_i$ 
  there exists at most one $v\in \mathfrak t$ with $\|v\|\le \varepsilon$ and $y\in \exp (v)Of_iY$. 
For any $y\in X$ and $1\le i\le k$, we set 
\[
\beta_i(y)=\left\{\begin{array}{cl}
\|v\|^{-\vartheta} & \mbox{if }y\in \exp(v)Of_i Y\cap N X_{l_0} \mbox{ with } v \in \mathfrak t \mbox{ and } \|v\|\le\varepsilon \\
\varepsilon ^{-\vartheta} & \mbox{otherwise}. 
\end{array}
\right .
\]
We let  $\beta(y)=\beta_1(y)+\cdots+\beta_k(y)$ 
which has  the following properties:
\begin{enumerate}[(I)]
\item $\beta$ is lower semicontinous. 
\item $\beta $ is Lipschitz with respect to the action of  $G$ on $NX_{l_0}$, i.e.~for every 
 $g\in G$ and $y\in  NX_{l_0}$ one has  
 \begin{align}
 \label{eq;lip}
  \beta(gy)\le\max\{ \|g^{-1}\|, \|g\| \}^\vartheta \beta(y).
 \end{align}
\item $\beta(y)=\infty$ if and only if $y\in NX_{l_0}\cap (\cup_{i=1}^k Of_i)Y$.
\end{enumerate}
Note that
(\ref{eq;lip})  holds with $\beta$ replaced by $\beta_i$. Also, $\beta(x)<\infty$
since $Gx$ is dense in $X$.

The value of $\beta$ on $X\setminus  NX_{l_0}$ will not  play an important role in the proof, since we 
will consider the first return cocycle to $X_{l_0}$. 
Our strategy is in principle the same as that of the  previous section. The key ingredient is 
Lemma \ref{lem;new eq} which is a variant of Lemma \ref{lem;basic ineq}.

\begin{lem}\label{lem;lcontract}
Let 	$r: I^m\to \mathbb N\setminus 0$ be a bounded  measurable function. 
Then there exists $b_0>0$ depending on 
the upper bound of the function $r$ such that 
for every $y\in N X_{l_0}$  
\[
\int_{I^m}\beta(g_{r(w)t} u(w)y) \dd w\le  e^{-t \vartheta \lambda_0 } \beta (y)+b_0.
\]
\end{lem}
\begin{proof}
Since  the function $r$ is bounded, there exists $C\ge 1$ such  that 
$\max\{\|g_{r(w)t}u(w)\|,\|(g_{r(w)t}u(w))^{-1}\| \}\le C$ for every $w\in I^m$. 
For $1\le i\le k$, let  $$J_i=\{w\in I^m: \beta_i (g_{r(w)t}u(w)y)\ge (C\varepsilon^{-1} )^{\vartheta} \}$$ 
and $J'_i=I^m \setminus J_i$.
If $w\in J_i$, then $y=\exp({v_i}) Of_i Y$ with $\beta_i (y)=\|v_i\|^{-\vartheta}$ and
 $\beta_i (g_{r(w)t}u(w)y)=\|g_{r(w)t}u(w)v_i\|^{-\vartheta}$. So by  (\ref{eq;lcon})
\[
\int_{J_i}\beta_i(g_{r(w)t} u(w)y)\dd w  \le\int_{I^m} \|g_{r(w)t}u(w)v_i\|^{-\vartheta}\dd w \le e^{-t\vartheta \lambda_0}\beta _i(y).
\]
By the splitting $\int_{I^m} \beta_i=\big(\int_{J_i}+\int_{J_i'}\big)\beta_i$, one has
\[
\int_{I^m}\beta_i(g_{r(w)t} u(w)y)\dd w\le e^{-t\vartheta \lambda_0}\beta _i(y)+
2^m(C\varepsilon^{-1} )^{\vartheta}. 
\] 
Since $\beta=\sum_{i=1}^k \beta_i$, 
the lemma holds by taking 
$b_0=2^mk(C\varepsilon^{-1} )^{\vartheta}$.
\end{proof}

\begin{lem}\label{lem;new eq}
There exists  $b_0>0$ such that  for  any  $n\in \mathbb N$, $w_0\in I^m$ with $\sigma_n(w_0)<\infty$
 and $J=I(\sigma_n,w_0)$ one has 
\begin{equation}\label{eq;goalsingular}
\int_{J}\beta (g_{\sigma_{n+1}(w)t}u(w)x)\dd w\le e^{-t\vartheta \lambda_0/2} \int_{J}
\beta(g_{\sigma_n(w)t}u(w)x)\dd w +b_0|J|.
\end{equation}
\end{lem}
Remark: Recall that  $\mathcal C_n$ is the smallest sigma algebra of $I^m$ generated by   $I(\sigma_i, w)$  for
  $0\le i\le n$ and $w\in I^m$ with $\sigma_i(w)<\infty$.
Since modulo null sets every element  of $\mathcal C_n$ is a disjoint  union of at most countably many   sets of the form $I(\sigma_n, w)$, the lemma
also holds for $J\in \mathcal C_n$.

\begin{proof}
Since the function $\sigma_n(w)$ is fixed on $J$ we simply write $\sigma_n$ for $\sigma_n(w)$.
Here  $\sigma_{n+1}(w)-\sigma_n$ varies for different $w$ and might be unbounded, so we can not
 use the  idea of shadowing Lemma \ref{lem;shade} 
directly. To overcome this difficulty we fix a positive integer $Q$ which will be specified 
afterwards  and 
 define a truncation (and extend it to all the $\R^m$)   for the function $\sigma_{n+1}(w)-\sigma_n$ by 
\[
 r(w)=
\left\{
\begin{array}{cl}
\sigma_{n+1}(w)-\sigma_n & \mbox{if } w\in J \mbox{ and } \sigma_{n+1}(w)-\sigma_n<  Q \\
Q & \mbox{for all other } w\in \R^m.
\end{array}
\right.
\]

Note that for all $w\in J$, the point $\beta(g_{\sigma_nt}u(w)x)\in NX_{l_0}$.
So the  Lipschitz property  of $\beta$ in (\ref{eq;lip}) implies that there  exists
$C\ge 1$ such that  
\begin{equation}\label{eq;proof1}
C^{-1}\beta(g_{\sigma_nt}u(w_0)x)\le \beta(g_{\sigma_nt}u(w)x)\le C \beta(g_{\sigma_nt}u(w_0)x) 
\end{equation}
 for all 
$w\in J$.
 We take $Q $ to be a positive  integer such that 
\begin{equation}\label{eq;upperQ1}
C^2C_0\frac{ e^{-Qt(\lambda/2-\vartheta'\vartheta)}} {1-e^{-t(\lambda/2-\vartheta'\vartheta)}}
+e^{-t\vartheta\lambda_0 }  \le e^{-t\vartheta \lambda_0/2}
\end{equation}
where $C_0\ge 1$ is the constant satisfying (\ref{eq;lhard}). 
Let $b_0>0$ be the constant given by Lemma \ref{lem;lcontract} with respect
 to the truncated   function $r$. 
 Note that $C$ and hence $b_0$ does not depend on $n$ or $w_0$.

We divide $J$ into two sets: 
 \[
J_1=\{w\in J:  r(w)
<  Q\}\quad 
\mbox{and}\quad J_2=J\setminus J_1=\{w\in J:  r(w)= Q\}.
\quad
\]
Then \begin{align}\label{eq;tough}
 \int_{J_1}\beta (g_{\sigma_{n+1}(w)t}u(w)x)\dd w       \le  \int_{J}\beta (g_{r(w)t+\sigma_nt}u(w)x)\dd w.
 \end{align}
Let $w_n'=(e^{-\sigma_ntb_1}, \ldots, e^{-\sigma_ntb_m})$ where $b_i>0$ $(1\le i\le m)$
satisfy (\ref{eq;Recall}).
Similar to the proof of  Lemma \ref{lem;shade}, for every $\tilde w\in J$ the measure of   $J(\tilde w):=\{ w_1\in I^m: \tilde w-w_1 \cdot w_n'\in J \}$ (here $w_1\cdot w_n'$ 
is  the usual inner product on $\mathbb R^m$)
is at least one. So the right hand side of (\ref{eq;tough}) is  
\begin{eqnarray*}
 &\le& 
 \int_{J}\int_{I^m}\beta(g_{ r( w+w_1 \cdot w_n')t+\sigma_nt}u(w+w_1 \cdot w_n')x)\dd w_1\dd w\\
 &=& 
\int_{J}\int_{I^m}\beta(g_{ r( w+w_1 \cdot w_n')t}u(w_1)g_{\sigma_nt}u(w)x)\dd w_1\dd w \\
  &\le&   e^{-t\vartheta \lambda_0}\int_{J}
\beta(g_{\sigma_nt}u(w)x)\dd w+b_0|J|,
\end{eqnarray*}
where in the last inequality we use
 Lemma \ref{lem;lcontract}.

For every integer $q\ge Q$, let  $B_q=\{w\in J: \sigma_{n+1}(w)-\sigma_n=q\}$.
In view of  (\ref{eq;expgrow}) and the Lipschitz property  of $\beta$ in (\ref{eq;lip})  we have
\begin{eqnarray*}
\int_{J_2}\beta (g_{\sigma_{n+1}(w)}u(w)x)\dd w & \le &  
\sum _{q\ge  Q}\int_{B_q}e^{q\vartheta'\vartheta t} 
\beta(g_{\sigma_{n}t}u(w)x)\dd w \\
\mbox{by (\ref{eq;proof1})}\qquad& \le & C  
\sum _{q\ge  Q}\int_{B_q}e^{q\vartheta'\vartheta t} 
\beta(g_{\sigma_{n}t}u(w_0)x) \dd w \\
\mbox{by (\ref{eq;lhard})}\qquad
&\le & CC_0\sum_{q\ge Q} e^{-qt(\lambda_0/2-\vartheta'\vartheta)}\cdot\int_J\beta(g_{\sigma_nt}u(w_0)x)\dd w \\
\mbox{by (\ref{eq;proof1})}\qquad
&\le & C^2C_0 \frac{ e^{-Qt(\lambda_0/2-\vartheta'\vartheta)}} {1-e^{-t(\lambda_0/2-\vartheta'\vartheta)}}\cdot
\int_J\beta(g_{\sigma_nt}u(w)x)\dd w.
\end{eqnarray*}
In view of  (\ref{eq;upperQ1}),
the lemma follows from the estimate of the left hand side of (\ref{eq;goalsingular})  on $J_1$ an $J_2$. 
\end{proof}

We fix a positive number $l$ so that 
\begin{align}\label{eq;wait}
&\frac{ b_0}{l}+e^{-t\vartheta \lambda_0/2}\le e^{-t\vartheta \lambda_0/4},\\
\label{eq;similarY}
&x\in X^Y_l \quad \mbox{where }X^Y_{l}=\{y\in  u([-2,2]^m)X_{l_0}: \beta(y)\le l\}.
\end{align}
Since $\beta$ is lower semicontinuous, the set $X^Y_l$ is a closed subset of 
$u([-2,2]^m)X_{l_0}$.  Hence,  $X^Y_l$ is a compact subset of $X$ with $X^Y_l\cap FY=\emptyset $.

The  $i$th return time to $X_{l}^Y$  is  the function $\kappa_i:I^m\to { \mathbb N}\cup \{\infty\}$ which will be defined 
inductively. 
For all $w\in I^m$ we let 
$\kappa_0(w)=0$ and $\eta_0(w)=0$.
Suppose $\kappa_n$ and $\eta_n$ have been defined for $n<i$. 
Then we define 
\begin{align}
\label{eq;label}
\eta_i(w)&=\inf \{ j> \eta_{i-1}(w): g_{\sigma_j(w)t}u(I(\sigma_j,w))x\cap  X^Y_l\neq \emptyset  \}, \\
\kappa _i(w)&= \left\{
\begin{array}{ll}
\sigma_{\eta_i(w)}(w) & \mbox{if }\eta_i(w)<\infty\\
\infty &\mbox{otherwise}. 
\end{array}
\right.
\end{align}
If $\kappa_i(w)<\infty$, then  $\eta_i(w)$ is the index  $j$  such that $\kappa_i(w)=\sigma_j(w)$. 
To simplify the notation,
we write 
 \[
I({\kappa_n}, w)=I_{{\kappa_n(w)}}(w).
\]

\begin{lem}\label{lem;morelem}
	There exists  positive numbers  $Q\in \Z$  and $\lambda_1$ such that for any inegers
	$i\ge 0, j>0 $  and  any  $w_0\in I^m$ with $\sigma_i(w_0)<\infty$
	the measure of the set 
	\begin{equation}\label{eq;indexplane}
	A_{i, j}(w_0)=\{w\in I(\sigma_i,w_0): \sigma_{i+j}(w)-\sigma_i(w)\ge Qj\}
	\end{equation}
	is at most  $e^{ -j\lambda_1}|I(\sigma_i,w_0)|$.
\end{lem}
\begin{proof}
	For fixed $i$ and $w_0$, we  use Lemma \ref{lem;main}  for $W=I(\sigma_i,w_0), \mu=\frac{1}{|I(\sigma_i,w_0)|}\mathrm{Leb}, \xi_j=\sigma_{i+j}, \mathcal F_j=\mathcal C_{i+j}$ and $\varepsilon=1$.
 The assumption (\ref{eq;basic exp}) holds by 
	(\ref{eq;lhard}). 
	So there exist positive numbers  $\lambda_1$ and  $Q\in 2\Z$ 
	which do not  dependent on $w_0$ or $i$ in view of the remark of Lemma \ref{lem;main}  such that
	the measure of the set 
	\[
	A_{i,j}' (w_0):= \left \{w\in I(\sigma_i,w_0): \sum_{s=1}^j\mathbbm{1}_{\frac{Q}{2}}
	(\sigma_{i+s}(w)-\sigma_{i+s-1}(w))\ge j\right\} 
	\]
	is at most  $e^{-j\lambda_1}|I(\sigma_i,w_0)|$.
	Suppose that $w\in A_{i,j}(w_0)$ for some 
	$w\in I(\sigma_i,w_0)$ and $\sigma_{i+j}(w)<\infty$, we claim that  $w\in A'_{i,j} (w_0)$.
	If not, then
	\begin{align*}
	\sigma_{i+j}(w)-\sigma_i(w)&=\sum_{s=1}^{j}\sigma_{i+s}(w)-\sigma_{i+s-1}(w)\\
	& <  j+\frac{Qj}{2} \le Qj,
	\end{align*}
	which contradicts $w\in A_{i,j}(w_0)$. 
	Therefore,  $|A_{i, j}(w_0)|\le |A'_{i, j}(w_0)|\le e^{-j\lambda_1}|I(\sigma_i,w_0)|$. 

\end{proof}

\begin{lem}\label{lem;new exp}
There exists $C_1\ge 1$ such that for any $ n,q\in \mathbb N$ with $q>0$  and $w_0\in I^m$
with $\kappa_n(w_0)<\infty$ the measure of the  set 
\[
B_{n,q}(w_0)=\{w\in I({\kappa_n}, w_0): \eta_{n+1}(w)-\eta_{n}(w)\ge q \} 
\]
is at most 
$ C_1e^{-qt\vartheta\lambda_0/4}|I({\kappa_n},w_0)|$.
\end{lem}
\begin{proof}
We fix $n, w_0$ and set  $B_q =B_{n, q}(w_0), i=\eta_{n}(w_0), \sigma_i=\sigma_i(w_0)$. Let 
\[
s_q=\int_{B_{q+1}}\beta(g_{\sigma_{{i+q}}(w)t}u(w)x)\dd w \qquad (q\ge 0).
\] 
Note that for all $w\in B_{q+1}\ (q\ge 1) $ with $\sigma_{i+q}(w)<\infty$ one has  $\beta(g_{\sigma_{i+q}(w)t}u(w)x)>l$.  So 
\begin{align}\label{eq;key1}
s_{q}\ge l |B_{q+1}|\quad  \mbox{ for all } q\ge  1.
\end{align}

Note that $B_{q}\in\mathcal C_{i+q-1}$ for $q\ge 1$,  so Lemma \ref{lem;new eq} implies that
 \begin{align}\label{eq;key}
s_q &\le \int_{B_{q}}\beta(g_{\sigma_{i+q}(w)t}u(w)x)\dd w \\
& \le  e^{-t\vartheta \lambda_0/2}\int_{B_{q}}\beta(g_{\sigma_{i+q-1}(w)t}u(w)x)\dd w+b_0|B_{q}|.\notag
\end{align}

By (\ref{eq;wait}), (\ref{eq;key1}) and   (\ref{eq;key})  one has
$
s_q\le e^{-t\vartheta\lambda_0/4}s_{q-1}
$
for all $q\ge 2$.
An induction on $q$ for $q\ge 2$ and (\ref{eq;key}) for $q=1$ gives
\begin{align}
\label{eq;key2}
s_q\le e^{-(q-1)t\vartheta\lambda_0/4} s_1\le  e^{-qt\vartheta\lambda_0/4}  s_0+b_0
e^{-(q-1)t\vartheta\lambda_0/4}|B_1|.
\end{align}
Note that $B_1= I(\sigma_{i}, w_0)$ and there exists $w'\in I(\sigma_{i}, w_0)$ with 
$$\beta(g_{\sigma_{i}t}u(w')x)\le l.$$ 

 By the Lipschitz property (\ref{eq;lip}), there exists $C\ge 1$ such that 
\begin{align}\label{eq;key3}
s_0=\int_{I(\sigma_{i}, w_0)} \beta(g_{\sigma_{i}t}u(w)x)\dd w\le C l |I(\sigma_{i}, w_0)|.
\end{align}
So by (\ref{eq;wait}), (\ref{eq;key1}), (\ref{eq;key2}) and (\ref{eq;key3}), 
for all $q\ge 1$ 
\[
|B_{q+1}|\le (C+ e^{t\vartheta \lambda_0/4}b_0/l) e^{-qt\vartheta\lambda_0/4} |I(\sigma_{i}, w_0)|\le 2C e^{-qt\vartheta\lambda_0/4} |I(\sigma_{i}, w_0)|.
\]
So the conclusion holds for $C_1=2C e^{t\vartheta\lambda_0/4}$.

\end{proof}

\begin{lem}\label{lem;combine}
There exist positive numbers  $\lambda $  and $C\ge 1$ such that for all $n\in \mathbb N, q\ge 1$ and $w_0\in I^m$ with $\kappa_n(w_0)<\infty$
the measure of the  set 
\begin{equation}\label{eq;really need}
J_{n,q}'(w_0)=\{w\in I({\kappa_n},w_0): {\kappa_{n+1}(w)}-{\kappa_n(w)}\ge q\}
\end{equation}
 is at most 
$C e^{-q\lambda}|I({\kappa_n},w_0)|$.
\end{lem}
\begin{proof}
Let $Q\in \Z, \lambda_1, C_1\ge 1$ are positive numbers so that Lemmas \ref{lem;morelem} and \ref{lem;new exp} hold. 
We show that for  $\lambda=\frac{1}{Q}\min\{\lambda_1, t\vartheta \lambda_0/4  \}$ and  $C=2C_1 e^{Q\lambda}$ 
\begin{align}\label{eq;to finish}
|J_{n,q}'(w_0)|\le C e^{-q\lambda}|I({\kappa_n},w_0)|.
\end{align}
We fix $w_0,n$ and write $i=\eta_n(w_0)$. 
Note that $Ce^{-q\lambda}\ge 1$ for all $q< Q$.
So (\ref{eq;to finish})  holds trivially for $q<Q$. Now assume $q\ge Q$. 
Recall that $\lfloor q/Q\rfloor$ is the largest integer less than or equal to $q/Q$. 
We claim that \[
J_{n,q}'\subset A_{i, \lfloor q/Q\rfloor }\cup B_{n, \lfloor q/Q\rfloor}.
\]

Suppose $w\in I(\kappa_n, w_0) $ but  $w\not\in  B_{n, \lfloor q/Q\rfloor}(w_0)$, then $\eta_{n+1}(w)-\eta_n(w)< 
\lfloor q/Q\rfloor$. If we also have $w\not \in A_{i, \lfloor q/Q\rfloor }(w_0)$, then 
$$\sigma_{i+\lfloor q/Q\rfloor}(w)-\sigma_i(w)< Q \lfloor q/Q\rfloor.$$  
According to the relation between $\kappa$ and $\eta$, 
\begin{align*}
\kappa_{n+1}(w)-\kappa_n(w)&=\sigma_{\eta_{n+1}(w)}(w)-\sigma_{\eta_n(w)}(w)\\
&< \sigma_{i+\lfloor q/Q\rfloor}(w)-\sigma_i(w)\\
&< \lfloor q/Q\rfloor Q\le q,
\end{align*} 
which implies $w\not\in J_{n,q}'(w_0)$. This completes the proof of the claim. 

Therefore, by Lemmas \ref{lem;morelem} and \ref{lem;new exp},
\begin{align*}
|J_{n,q}'(w_0)|&\le  |A_{i, \lfloor q/Q\rfloor }(w_0)|+ |B_{n, \lfloor q/Q\rfloor}(w_0)|\\
&\le\big (e^{-\lfloor q/Q\rfloor  \lambda_1 }+C_1e^{-\lfloor q/Q\rfloor t\vartheta\lambda_0/4}\big)|I({\kappa_n},w_0)|\\
&\le C e^{-q\lambda  }|I({\kappa_n},w_0)|.
\end{align*}

\end{proof}

\begin{proof}[Proof of Proposition \ref{prop;singular}]
Let $\varepsilon_0$ be an arbitrary number with $1<\varepsilon_0<1/2$. 	
Let $l$ be a positive number such that (\ref{eq;wait}) and  (\ref{eq;similarY}) hold.
Recall that  $X_l^Y$ is the compact subset of $X$ defined in (\ref{eq;similarY})
 and
$\kappa_n\ (n\in \N)$ is the returning function to $X_l^Y$. 
Since $X_l^Y\cap FY=\emptyset$, so does $GX_l^Y\cap FY$.

Let $\mathcal F_i\ (i\in \N)$ be the sigma algebra on $I^m$ generated by $I(\kappa_n, w)$
for all $w\in I^m$ and $0\le n\le i$ with $\kappa_n(w)<\infty$. 
 Lemma \ref{lem;combine}  implies that the assumption of  Lemma \ref{lem;main} holds
 for $W=I^m, \mu=\frac{1}{2^m}\mathrm{Leb}, \xi_n=\kappa_n$ and the filtration $(\mathcal F_i)$. 
So  there exist  positive numbers
 $a_0<1$ and  $Q$  such that
the measure of the set 
\[
J_n=
\left \{w\in I^m: \frac{1}{n} \sum_{i=1}^n: \mathbbm{1}_Q
({\kappa_i(w)}-{\kappa_{i-1}(w)})\ge {\varepsilon_0}
\right \}
\]
is at most  $2^ma_0^n$. 
We claim that for 
$
K_0=\bigcup_{0\le s\le Qt}g_su([-2, 2]^m)X_l^Y$ and every positive integer $n$
\[
|\{w\in I^m: \mathfrak D^n_{K_0}(w)\le 1-\varepsilon_0  \}|\le 2^m a_0^n. 
\]
The proof of the claim is the same as the proof given in Lemma \ref{lem;discrete}, 
so we refer the readers there for details. 
By Lemma \ref{lem;continuousd}, there exist positive numbers $a<1$, $C\ge 1$
and a compact subset $K\subset GK_0\subset GX_l^Y$ such that for all $T>0$
\[
|\{w\in I^m:  \mathcal A_K^T(w)\le 1-2\varepsilon_0\}|\le Ca^T. 
\]
\end{proof}

\appendix
\section{}

Let $G, g_t, G^+$ be as  in  Theorem \ref{thm;1}. 
In this section we give two more  characterizations of $g_1$ expanding 
subgroups and  prove  Lemma \ref{lem;abelian}. 
\begin{lem}\label{lem;characterize}
Let $U$ be a connected  $\Ad$-unipotent subgroup of $G$ normalized 
by $\{g_t:t\in \mathbb R\}$. The following statements are equivalent:
\begin{enumerate}[\rm  (1) ]
	\item $U$ is $g_1$ expanding.
	\item For any nontrivial irreducible     representation $\rho:
	G\to \GL(V) $, the subspace of $U$-fixed vectors $V^U:=\{v\in V: \rho(u)v=v  \}$ is contained in $V^+$. 
\end{enumerate}
\end{lem}

\begin{proof}
 It follows easily from   the definitions that 	$(1)\Rightarrow(2)$.
  Now  we show  that $(2)\Rightarrow (1)$. 
 Let $v$ be a   nonzero vector   
 in  $V$. 
 Since $g_t\ (t\in \R)$ normalizes $U$, we have $\rho(g_t) V^U=V^U$. 
 Since $g_t$ is $\Ad$-diagonalizable, there is a 
 $\{g_t: t\in \mathbb R\}$ invariant subspace $W$   complementary to   $V^U$. 
 Let $\pi': V\to V^U$ be the  projection with respect to $W$. 
 It follows from \cite[Lemma 5.1]{s96}  that $\pi'(\rho(U)v) \neq \{0 \} $.
 Since $V^U \le V^+$ by (2), 
  the group $U$ is  $g_1$ expanding by definition.
 \end{proof}

\begin{lem}\label{lem;characterize+}
Let $U$ be a connected and closed  subgroup of $G$ normalized 
by $\{g_t:t\in \mathbb R\}$.
Then 
	 $U$ is $g_1$ expanding if and only if 
 $U\cap G^+$ is $g_1$ expanding. 
 \end{lem} 
Remark:
Note that $U\cap G^+$ is connected since it is  normalized by $\{g_t: t\in \R \}$. Indeed, for any   $u\in U\cap G^+$, the element
$g_{-t} u g_{t}$ belongs to the connected component of the identity of  $ U\cap G^+$
for $t$ sufficiently large. 

 \begin{proof}
 It is clear from the definitions that if  $U\cap G^+$  is $g_1$ expanding, then so
 is $U$.	
 
 We prove the other direction by contradiction. Assume now that $U$ is $g_1$
 expanding but $U\cap G^+$ is not. Then  there exists a nontrivial irreducible representation  $\rho: G\to \GL(V)$ and a nonzero vector $v\in V$ such that 
 $\pi_+(\rho(u)v)$  is zero for all $u\in U\cap G^+$. In other words, we have 
 $\rho( U\cap G^+)v\subset V^0\oplus V^-$. Let $G^{0-}\le G$  be the connected subgroup invariant under the conjugation of $\{g_t: t\in \R \}$ such that $G^{0-}\cap G^+$
 is the identity element. Then the subspace $V^0\oplus V^-$ is   $\rho(G^{0-})$ invariant. It follows that  $\rho(U\cap G^{0-})\rho( U\cap G^+)v\subset  V^0\oplus V^-$ and hence there is an open neighborhood   $N$ of  the identity element of $U$ such that $\rho(N)v\subset V^0\oplus V^-$. Note that $\GL(V)$ has  a natural structure
 of real algebraic groups and $\{  g\in \GL(V) : \pi_+(gv)=0\}$ is Zariski closed. 
 Also, note that the   Zariski closure of $\rho(N)$ contains $\rho(U)$ since $U$ is  connected. So $\rho(U)v 
 \subset V^0\oplus V^-$ which contradicts the assumption that $U$ is $g_1$ expanding.

 \end{proof}

The key ingredient of the proof of Lemma \ref{lem;abelian} is the following 
result about abstract root systems.

\begin{lem}\label{lem;root}
Let $\Phi$ be an irreducible abstract root system  and let
 $E=\mathrm{span}_{\mathbb R}\Phi $.
 Suppose that $E$  has dimension $n$ and 
 $\langle\cdot, \cdot\rangle $ is the inner product of $E$ invariant under the Weyl group of $\Phi$.    Let $\Phi^+\subset \Phi$ be a positive system dominated by 
some $\alpha\in E$, i.e.~$\langle \alpha, \beta\rangle\ge 0 $ for any $\beta\in \Phi^+$. Then there exists a basis  $\beta_1, \ldots, \beta_n\in \Phi^+$ of $E$   such that 
\begin{equation}\label{eq;linearly}
\alpha=c_1\beta_1+\cdots +c_n\beta_n
\end{equation}
where  $c_i\ge 0$ and $\beta_i+\beta_j\not\in \Phi^+$ for any $i,j$.
\end{lem}

\begin{proof}
The only 
irreducible nonreduced  root systems are of types $(BC)_n$ $ {(n\ge 1)}$, see e.g.~\cite[\S II.8]{knapp}. If we take the  subsystem of $(BC)_n$ consisting of all the roots $\beta$
with $2\beta$ not a root, then it is a root system of type 
$C_n $ if $n\ge 3$, or $B_2$ if $n=2$,  or $A_1$ if $n=1$. 
So it suffices to prove the lemma for reduced   $\Phi $ which we assume now.

Recall that two roots $\beta$ and $\gamma$ are said to be  strongly
orthogonal if neither one of  $\beta \pm \gamma$ is a root and 
a subset $\mathcal O$ of $\Phi^+$ is called strongly orthogonal system if elements of $\mathcal O$
are pairwise strongly orthogonal. 
It follows from Oh \cite{oh98} that
if $\Phi$ is of  type 
 $B_n\ (n\ge 2), C_n \ (n\ge 3), D_n\  (n \mbox{ is even and } n\ge 4),  E_7, E_8, F_4, G_2, $ then
 there is a strongly orthogonal system $\mathcal O$  consisting of $n$ elements.
 In these cases
$\alpha$ is a linear combination
 of elements in $\mathcal O$ satisfying the conclusion  of the lemma. 

Now we assume $\Phi$
is of type $A_n, D_n$ or $ E_6$.
Let $\|\cdot\|$ be the induced norm on $E$.
 Let $\Pi=\{\alpha_1, \ldots, \alpha_n\}$ be simple roots 
 determined by  $\Phi^+$  and let $A$ be the associated  Cartan matrix.   We assume without loss of generality that 
 $\|\alpha_i\|=1$.
It follows from  Lusztig and Tits \cite{lt92} that $A^{-1}$ has positive rational entries. So  we have
\[
\alpha=a_1\alpha_1+\cdots+a_n\alpha_n
\]
where $a_i\in \mathbb R_{>0}$.

{\bf Case}   $\mathbf {A_n}$.
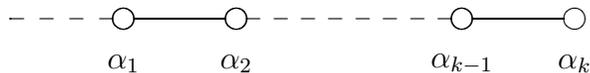
\[ 
\xy 
(-15,0)*{\cir<0pt>{}}; (0,0)*{\cir<4pt>{}}; **\dir{--};
(0,0)*{\cir<4pt>{}}; (15,0)*{\cir<4pt>{}}; **\dir{-};
(15,0)*{\cir<4pt>{}}; (45,0)*{\cir<4pt>{}}; **\dir{--};
(45,0)*{\cir<4pt>{}}; (60,0)*{\cir<4pt>{}}; **\dir{-};
(0,-6)*{\alpha_1};
(15, -6)*{\alpha_{2}};
(45, -6)*{\alpha_{k-1}};
(60, -6)*{\alpha_{k}};
\endxy 
\]
We assume that the simple roots $\Pi$ are indexed so that  $a_1\ge a_2$
and the corresponding Dynkin  diagram is as above. 
  Since $\alpha$
is dominated we have
\[
\langle \alpha, \alpha_2\rangle=a_2-\frac{1}{2}a_1-\frac{1}{2}a_3\ge 0
\]
which implies  $a_2\ge a_3$. The same argument using $\langle\alpha, \alpha_i\rangle\ge 0$
and $a_{i-1}\ge a_i$ for $2\le i<k$ inductively  imply  $a_i\ge a_{i+1}$ for
all $1\le i\le k-1$. 
Therefore
 $\Pi$ can be rearranged so that for all $1\le i<n$ one has   $a_i\ge a_{i+1}$  and $\alpha_{i+1}$ is connected to one
 of $\{\alpha_1, \ldots, \alpha_i\}$ in the Dynkin diagram.  
So  if we take  $\beta_i=\alpha_1+\cdots+\alpha_i\in \Phi ^+$ the conclusion of the lemma holds .

 {\bf Case $\mathbf {D_n}$ where $n\ge 4$ is odd}. 
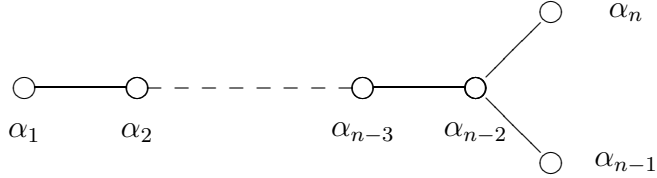
\footnote{In this case the strongly orthogonal $\mathcal O$ constructed in \cite{oh98} contains 
$n-1$ elements (say $\beta_1, \ldots, \beta_{n-1}$) and the highest root (say $\beta_n$) is not in $\mathcal O$. 
These $n$ elements satisfy 
 $\beta_i+\beta_j\not \in \Phi$ but it is not clear to 
the author how to prove $c_i\ge 0$.}

\[ 
\xy 
(0,0)*{\cir<4pt>{}}; (15,0)*{\cir<4pt>{}}; **\dir{-};
(15,0)*{\cir<4pt>{}}; (45,0)*{\cir<4pt>{}}; **\dir{--};
(45,0)*{\cir<4pt>{}}; (60,0)*{\cir<4pt>{}}; **\dir{-};
(70,10)*{\cir<4pt>{}}; (60,0)*{\cir<4pt>{}}; **\dir{-};
(70,-10)*{\cir<4pt>{}}; (60,0)*{\cir<4pt>{}}; **\dir{-};
(0,-6)*{\alpha_1};
(15, -6)*{\alpha_{2}};
(45, -6)*{\alpha_{n-3}};
(60, -6)*{\alpha_{n-2}};
(80, -10)*{\alpha_{n-1}};
(80, 10)*{\alpha_n};
\endxy 
\]
We  assume that $\Pi$ is indexed so that the Dynkin diagram is as above
and $a_{n-1}\ge a_n$.
There is 
an explicit list of   $\Phi^+ $ and  $\Pi$ with $E=\mathbb R^n$  given in 
Knapp  \cite[Appendix C]{knapp} 
as follows:
$\alpha_i=\mathbf e_i- \mathbf e_{i+1}$ for $i<n$ and $\alpha_n=\mathbf e_{n-1}+\mathbf e_n$; 
$\Phi^+=\{ \mathbf  e_i\pm \mathbf  e_j: i<j\}$ where $\{\mathbf e_1, \ldots, \mathbf  e_n \}$ is
the standard basis of $\mathbb R^n$.

Since $\alpha $ is dominated we have 
 \[
2\langle \alpha, \alpha_1\rangle =2a_{1}-a_{2}\ge 0.
\]
Assume that $(i+1)a_i-ia_{i+1}\ge 0$ for $i\le n-4$. Then
\[
(i+2)a_{i+1}-(i+1)a_{i+2}=2(i+1)\langle \alpha, \alpha_{i+1}\rangle +(i+1)a_i-ia_{i+1}\ge 0.
\]
Therefore  
\begin{equation}\label{eq;mosquito}
(i+1)a_i-ia_{i+1}\ge 0 \quad \mbox{for } 1\le i\le n-3.
\end{equation}
 By calculating  inner products of $\alpha$ with $\alpha_{n-1}$ and $\alpha_n$ we have
 \begin{equation}\label{eq;coefficient}
 \left\{
\begin{array}{lcl}
2a_{n} & \ge & {a_{n-2}} \\
2a_{n-1} & \ge &{a_{n-2}} .
\end{array}
\right.
\end{equation}
It follows form (\ref{eq;coefficient})  and 
\begin{align}\label{eq;readable}
2\langle\alpha, \alpha_{n-2}\rangle=2a_{n-2}-a_{n-3}-a_{n-1}-a_n \ge 0 
\end{align}
that
$
a_{n-3}\le a_{n-2}.
$
Starting from this and using  $\langle\alpha, \alpha_{n-2-i}\rangle \ge 0$ inductively  for $1\le i\le n-4$ one gets $a_{n-2-i}\ge a_{n-3-i}$ {for }
$0\le i\le 4$. That is
\begin{equation}\label{eq;coefficient1}
a_i \le a_{i+1} \quad \mbox{for } 1\le i\le n-3.
\end{equation}

For $1\le i\le  (n-3)/2$  we take 
 \[\left\{
\begin{array}{lcl}
\beta_{2i-1} &=& \alpha_{2i-1}\\
\beta_{2i} &=& \alpha_{2i-1}+2(\alpha_{2i}+\cdots+\alpha_{n-2})+\alpha_{n-1}+\alpha_n.
\end{array}
\right.
\]
It follows from (\ref{eq;mosquito}), (\ref{eq;coefficient}) and (\ref{eq;coefficient1})  that there are nonnegative  integers $c_1, \ldots, c_{n-3}, b_{n-2}, b_{n-1}, b_n $ (e.g.~ $c_2=a_2/2$ and $c_1=a_1- c_2$)  such that \[
\alpha-c_1\beta_1-\cdots-c_{n-3}\beta_{n-3}=b_{n-2}\alpha_{n-2}+b_{n-1}\alpha_{n-1}+b_n\alpha_n,
\]
where 
\begin{align}
\label{eq;readable1}
b_{n-2}=a_{n-2}-a_{n-3}, 
b_{n-1}=a_{n-1}-\frac{a_{n-3}}{2}, b_n=a_n-\frac{a_{n-3}}{2}.
\end{align}
Note that $b_{n-1}\ge b_n$ since we assume $a_{n-1}\ge a_n$. 
Also, by (\ref{eq;readable}) and (\ref{eq;readable1})
\[
(b_{n-2}-b_{n-1})+(b_{n-2}-b_n)\ge 0.
\]
So $b_{n-2}\ge b_n $.

If $b_{n-2}\ge b_{n-1}\ge b_n$ we take
$\beta_{n-2} =  \alpha_{n-2}+\alpha_{n-1}+\alpha_n$,
 $\beta_{n-1}=\alpha_{n-2}+\alpha_{n-1}$ and $\beta_n=\alpha_{n-2}$.
 Then (\ref{eq;linearly}) holds with   $c_{n-2}=b_n ,c_{n-1}=b_{n-1}-b_n, c_n=b_{n-2}-b_{n-1}$.
The fact that 
$\beta_i+\beta_j\not \in \Phi^+$ follows from  the list of $\Phi^+$
and 
 \[
 \left\{
\begin{array}{lcll}
\beta_{2i-1} &=&  \mathbf e_{2i-1}- \mathbf e_{2i} &  \mbox{for }1\le i\le  (n-3)/2\\
\beta_{2i} &=&  \mathbf e_{2i-1}+ \mathbf e_{2i}               & \mbox{for }1\le i\le  (n-3)/2\\
\beta_{n-2}& = &  \mathbf e_{n-2}+ \mathbf e_{n-1} &  \\
\beta_{n-1}& = &  \mathbf e_{n-2}- \mathbf e_{n}  & \\
\beta_{n}& = &  \mathbf e_{n-2}- \mathbf e_{n-1}.  & 
\end{array}
\right.
\]

Similarly, if $b_{n-1}> b_{n-2}\ge b_n$, then 
$\beta_{n-2} =  \alpha_{n-2}+\alpha_{n-1}+\alpha_n$, $\beta_{n-1}=\alpha_{n-2}+\alpha_{n-1}$ and $\beta_n=\alpha_{n-1}$
works.

 {\bf  Case $\mathbf {E_6}$}.

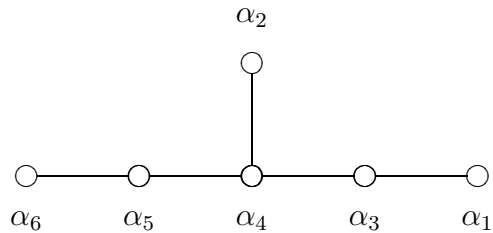
\[ 
\xy 
(30,0)*{\cir<4pt>{}}; (45,0)*{\cir<4pt>{}}; **\dir{-};
(45,0)*{\cir<4pt>{}}; (60,0)*{\cir<4pt>{}}; **\dir{-};
(60,0)*{\cir<4pt>{}}; (75,0)*{\cir<4pt>{}}; **\dir{-};
(75,0)*{\cir<4pt>{}}; (90,0)*{\cir<4pt>{}}; **\dir{-};
(60,0)*{\cir<4pt>{}}; (60,15)*{\cir<4pt>{}}; **\dir{-};
(30, -6)*{\alpha_6};
(45, -6)*{\alpha_5};
(60, -6)*{\alpha_4};
(75, -6)*{\alpha_3};
(90, -6)*{\alpha_1};
(60,21)*{\alpha_2};
\endxy 
\]
We assume that $\Pi$ is indexed so that the Dynkin diagram is as above. 
 A simple calculation using  $\langle \alpha, \alpha_i\rangle \ge 0$ gives
\begin{equation}
\label{eq;finish}
\left\{
\begin{array}{rcl}
2a_1 &\ge  &a_3 \\
2a_6 & \ge & a_5  \\
3a_3 & \ge & 2a_4\\
3a_5 & \ge & 2a_4 \\
2a_4 & \ge & 3a_2.
\end{array}
\right.
\end{equation}
Let  $\beta_1 = \alpha_1+2\alpha_2+2\alpha_3+3\alpha_4+2\alpha_5+\alpha_6$ 
to be the highest root. 
Then it follows from (\ref{eq;finish}) that
there exist nonnegative numbers 
$b_1, b_3, b_4, b_5, b_6$
such that 
\[
\alpha-c_1\beta_1=b_1\alpha_1+b_3\alpha_3+b_4\alpha_4+b_5\alpha_5+b_6\alpha_6
\quad \mbox{where } c_1=\frac{a_2}{2}.
\]

Note that we don't need to worry about $\beta_1+\beta_j\in \Phi^+$ for any choice of $\beta_j$ since $\beta_1$ is the highest root. 
Next we take 
$\beta_2=\alpha_1+\alpha_3+\alpha_4+\alpha_5+\alpha_6$ and $c_2=b_i$ where  
 $b_i=\min\{b_1, b_3, b_4, b_5, b_6\}$.   Then
 $
 \alpha-c_1\beta_1-c_2\beta_2
 $
 is a positive  linear  combination of at most $4$ simple roots.

The next choice depends on $i$ but is simple due to the following observation:
if $\gamma=\sum_{j=1}^{6}l_j \alpha_j\in \Phi^+ $ and $l_j\ge 2$ for some $j$, then
$l_2>0$. For each  connected component $Y$ of $\Pi\setminus \{\alpha_2, \alpha_i \}$ in the Dynkin diagram,  we define 
$\beta_Y=\sum_{\gamma \in Y}\gamma \in \Phi^+$ and $c_Y$ to be the largest number $c$ such that $\alpha-c_1\beta_1-c_2\beta_i-c\beta_Y$ is a nonnegative linear combination 
of $\Pi$. Then 
\[
\alpha-c_1\beta_1-c_2\beta_i-\sum_Y c_Y\beta_Y. 
\]
is a positive linear combination of at most $3$ simple roots.

In the next step we 
  continue to define $\beta$ according to the connected components of 
  the above set of simple roots. 
  In this way we can find a set  $\{\beta_i :1\le i\le 6 \}\subset \Phi^+$ such that 
(\ref{eq;linearly}) holds for some $c_i\ge 0$. The property $\beta_i+\beta_j\not \in
\Phi^+$ follows from the observations above.  
\end{proof}

\begin{proof}[Proof of Lemma \ref{lem;abelian}]
Let $\mathfrak g^+\subset \mathfrak g$ be the Lie algebras of $G^+\le G$. There exists $z\in \mathfrak g$ such that $g_t =\exp tz$. 
Note that every semisimple subalgebra
$\mathfrak g_1$ of $\mathfrak g$ is the Lie algebra of a closed subgroup $G_1$.
So it suffices to find a semisimple subalgebra $\mathfrak g_1$ containing  $z $ and 
an abelian subalgebra $\mathfrak u\subset \mathfrak g_1$ with certain properties. 
Since the projection of $g_1$ to each simple factors of $G$ is not the 
identity element, we assume without loss of generality that $\mathfrak g$ is a 
simple noncompact Lie algebra.	
	
Let $\mathfrak a$ containing $z$ be a maximal $\mathbb R$ split Cartan  subalgebra of  $\mathfrak g$ 
contained in the normalizer of $\mathfrak g^+$. 
Let $\mathfrak a^*$ be the dual vector space of $\mathfrak a$ and let 
$\Phi=\Phi(\mathfrak g, \mathfrak a) \subset \mathfrak a^*$ be  the relative  root system.
We fix   a positive system    $\Phi ^+$ so that $\alpha(z)\ge 0$ for all  $\alpha\in \Phi ^+$.
Let $\mathfrak g_\alpha\subset \mathfrak g$ be the root space of $\alpha\in \Phi$.

Let  $\langle w, w'\rangle =-B(w, \theta w')$ be the usual inner product on $\mathfrak a$, where
$B$ is the Killing form  and $\theta$
is  the Cartan involution of $\mathfrak g$ with  $\mathfrak a$ belonging to the $-1$ eigenspace.  The inner product allows us to define an isomorphism between
$\mathfrak a$ and $\mathfrak a^*$ where
 every $w\in \mathfrak g$ is  sent to  $\alpha_w\in a^*$ such that 
$\alpha_w(w')=\langle w, w'\rangle $.

It follows from Lemma \ref{lem;root} that there are nonnegative real numbers  $c_1, \ldots, c_n$  
and  roots $\beta_1,\ldots,  \beta_n  \in \Phi^+$ $(n=\dim \mathfrak a)$  such that 
\[
\alpha_z=c_1\beta_1+\ldots+c_n\beta_n\quad \mbox{and }\quad
\beta_i+\beta_j\not\in \Phi \quad \forall i, j.
\] 
Let 
\begin{align}
\label{eq;v}
P=\{1\le i\le n: c_i>0  \}\quad \mbox{and}\quad 
z=\sum_{i\in P} c_i z_i
\end{align}
 where  $\alpha_{z_i}=\beta_i$.  

According to \cite[Proposition 6.52]{knapp}, 
for each $i\in P$ there exists $w_i\in \mathfrak g_{\beta_i}$   such that 
$
\mathfrak h_i:=\mathrm{span}_\R \{z_i, w_i, \theta(w_i)\}$  is isomorphic to 
  $\mathfrak {sl}_2$.  Let $\mathfrak g_1$ be the smallest  subalgebra containing 
  all the $\mathfrak h_i \ (i\in P)$ and let 
 $\mathfrak u=\mathrm{span}_{\R }\{w_i: i\in P \}$.  Since $\beta_i+\beta_j\not \in \Phi$ for all $i, j\in P$, one has 
  $\mathfrak u$ is an abelian subalgebra. In view of  $z_i\in \mathfrak g_1 \ (i\in P)$
  and (\ref{eq;v}), one has $z\in \mathfrak g_1$. 
  
  Now we prove  $\mathfrak g_1$ is semisimple. 
  Since $[\mathfrak h_i, \mathfrak h_i]=\mathfrak h_i$ for $i\in P$, one has $[\mathfrak g_1, \mathfrak g_1]$. 
   So by \cite[ \S I.6.2 Corollary 3]{gov},  $\mathfrak g_1$ is an algebraic subalgebra, i.e.~the Lie algebra of an algebraic 
   subgroup of $\GL(\mathfrak g)$ via the adjoint embedding of $\mathfrak g$.
    It is clear from the previous paragraph that $\theta(\mathfrak g_1)=\mathfrak g_1$
    since     $\theta(\mathfrak h_i)=
    \mathfrak h_i$ for all $i\in P$. 
    So  by \cite[\S VI.3 Theorem 3.6]{gov}, the Lie algebra $\mathfrak g_1$ is reductive. Hence  $\mathfrak g_1$
    is semisimple in view of  $[\mathfrak g_1,\mathfrak g_1]=\mathfrak g_1$. 

Let $G_1$ and $U_a$ be connected subgroups of $G$ with Lie algebras $\mathfrak g_1$
and $\mathfrak u$ respectively. It follows from definitions  that $U_a\le G^+$ is an abelian group.
Let $\rho : G_1\to \GL(V)$ be a nontrivial irreducible representation and let  $\rho: \mathfrak g_1\to \mathfrak{gl}(V)$ be  the induced representation of the  Lie algebra.  Note that 
\[
V^{U_a}=V^{\mathfrak u}:=\{v\in V: \rho(w)v=0\quad \mbox{for all }w \in \mathfrak u \}.  
\]
Let $v\in V^{\mathfrak u}$ be a simultaneous  eigenvector of $\rho(z_i)\ (i\in P)$ with eigenvalue $a_i$.
Since $\rho(w_i)v=0$, 
 the  representation theory  of $\mathfrak {sl}_2$ (see e.g.~\cite[\S I.9]{knapp}) 
implies that  $a_i\ge 0$ and equality holds if and only if $\rho(\mathfrak h_i)v=0$. 
Since $\rho$ is nontrivial, some  $a_i>0$.
Recall that   $z$ is a positive linear combination of $z_i$ by (\ref{eq;v}), so 
$\rho(z)v=av $  for some $a>0$.
Therefore 
 $V^{U_a}=V^{\mathfrak u}\subset V^+$. Hence $U_a\le G^+$ is a $g_1$ expanding abelian subgroup of $G_1$ 
by Lemma \ref{lem;characterize}.
\end{proof}

\end{document}